\documentclass[11pt,a4paper,reqno]{amsart}
\usepackage{color}
\usepackage[latin1]{inputenc}
\usepackage{amsmath,amsfonts,amsthm,epsfig,graphicx}
\usepackage{caption}

\def\N{\mathbb N}

\def\R{\mathbb R}

\def\e{\varepsilon}

\newtheorem{theorem}{Theorem}[section]

\newtheorem{proposition}[theorem]{Proposition}
\newtheorem{lemma}[theorem]{Lemma}
\newtheorem{corollary}[theorem]{Corollary}

\numberwithin{equation}{section}
\numberwithin{figure}{section}

\begin{document}

\title{$C^\infty$ regularity in semilinear free boundary problems}

\author{Daniel Restrepo}

\address{Department of Mathematics, Johns Hopkins University,  3400 N. Charles Street, Baltimore, MD 21218}
\email{drestre1@jh.edu}

\author{Xavier Ros-Oton}

\address{ICREA, Pg. Llu\'is Companys 23, 08010 Barcelona, Spain \& Universitat de Barcelona, Departament de Matem\`atiques i Inform\`atica, Gran Via de les Corts Catalanes 585, 08007 Barcelona, Spain \& Centre de Recerca Matem\`atica, Barcelona, Spain.}
\email{xros@icrea.cat}




\allowdisplaybreaks

\begin{abstract}
We study the higher regularity of solutions and free boundaries in the Alt-Phillips problem $\Delta u=u^{\gamma-1}$, with $\gamma\in(0,1)$.
Our main results imply that, once free boundaries are $C^{1,\alpha}$, then they are $C^\infty$. In addition $u/d^{\frac{2}{2-\gamma}}$ and $u^{\frac{2-\gamma}{2}}$ are $C^\infty$ too.\\

In order to achieve this, we need to establish fine  regularity estimates for solutions of linear equations with boundary-singular Hardy potentials $-\Delta v = \kappa v/d^2$ in $\Omega$, where $d$ is the distance to the boundary and $\kappa\leq\frac14$.
Interestingly, we need to include even the critical constant $\kappa=\frac14$, which corresponds to $\gamma=\frac23$.
\end{abstract}

\maketitle

\section{Introduction}

A classical problem in the Calculus of Variations concerns the study of minimizers of functionals of the type 
\[
J(u) := \int_{D} {\textstyle\frac{1}{2}}|\nabla u|^2 +F(u).
\]
with boundary data $u=g\geq 0$ on $\partial D$.
The potential $F$ is usually asked to be nondecreasing, and with $F(0)=0$.
Under these conditions, minimizers of $J$ are non-negative and satisfy (in an appropriate sense) the Euler-Lagrange equation
\begin{equation}\label{eq ELgeneral}
\Delta u = F'(u)\quad \mbox{in}\quad D.
\end{equation}
Two of the most famous and important cases are $F(u)=u_+$ (the obstacle problem), and $F(u)=\chi_{\{u>0\}}$ (the Alt--Caffarelli or Bernoulli free boundary problem).
In general, when $F\notin C^{1,1}$, the solution will develop free boundaries: there will be a region of $D$ in which $u=0$, and a region in which $u>0$.
The so-called free boundary is the interface $\partial\{u>0\}$.

In this paper we are interested in the more general family of problems 
\[F(u)={\textstyle\frac{1}{\gamma}}u^\gamma \qquad \mbox{with}\qquad \gamma\in(0,2),\]
which gives rise to the so-called Alt--Phillips problem, introduced in \cite{AP86}.
This family of problems obviously encompasses both the obstacle problem ($\gamma=1$), as well as the Alt--Caffarelli problem ($\gamma\to0$).

The regularity theory for both the obstacle and the Alt--Caffarelli problems is quite well understood, and in particular the free boundary is known to be $C^\infty$ outside a closed set of singular points \cite{PSU12, Vel23,FR22,CS05}.
Moreover, fine properties of the singular set have been established in \cite{FRS20,JS15,EE19,FZ21}.

The regularity theory for minimizers in the Alt--Phillips problem has been developed in the works \cite{AP86,Bon01,FY23,WY22,DS21}, and the main known results may be summarized as follows:
\begin{enumerate}
	\item Solutions $u$ are $C^{\frac{2}{2-\gamma}}$, and this is optimal.\footnote{Throughout the paper, we denote $C^\beta=C^{k,\alpha}$ with $k+\alpha=\beta$ and $\alpha\in(0,1]$, $k\in \mathbb N$.}
	\item Free boundary points can be split into two sets: regular points, and singular ponts.
	\item The regular set is an open subset of the free boundary, and it is $C^{1,\alpha}$.
	\item The singular set is closed, and:
	
	- When $\gamma\in(0,1)$ its Hausdorff dimension is at most $n-3$.
	
	- When $\gamma\in[1,2)$ it is $(n-1)$-rectifiable.
\end{enumerate}

Our main goal in this paper is to prove that the set of regular points is actually $C^\infty$, and that solutions enjoy a $C^\infty$ regularity estimate near regular points, namely $u/d^{\frac{2}{2-\gamma}}\in C^\infty$.

This has been only proved for $\gamma\in(1,2)$ by Fotouhi--Koch \cite{FK23} and very recently for $\gamma\in(\frac{2}{3},2)$ by Allen-Kriventsov-Shahgholian \cite{AKS}, and has been open for a very long time in the case $\gamma\in(0,\frac{2}{3}]$.
Our proof is completely different from that in \cite{FK23} and \cite{AKS}, and covers the whole range of exponents $\gamma\in(0,2)$.

As we will see, quite interestingly, we find that in terms of higher order regularity the problem behaves like the obstacle problem for $\gamma\geq\frac23$, while it is more similar to the Bernoulli problem for $\gamma<\frac23$.
This is quite new and surprising, since in almost all previous regularity results the only differences were in cases $\gamma\in (1,2)$ vs $\gamma\in(0,1)$.

\subsection{Main result}

Our main result can be stated as follows.
Recall that regular free boundary points are those at which the free boundary $\partial\{u>0\}$ is $C^{1,\alpha}$ in a neighborhood.

\begin{theorem}\label{thm 1}
Let $\gamma \in (0,2)$ and $0\leq u \in C^{\frac{2}{2-\gamma}}$ be any viscosity solution of
	\begin{equation}\label{eq alt-phil}
			\Delta u = u^{\gamma-1} \quad \textrm{in}\quad \{u>0\}\cap B_1.
	\end{equation}
Assume $x_0 \in \partial \{u>0\}\cap B_1$ is a regular free boundary point.
Then, 
\begin{itemize}
\item The free boundary $\partial \{u>0\}$ is $C^\infty$ in a neighborhood of $x_0$.

\item The functions $u/d^{\frac{2}{2-\gamma}}$ and  $u^{\frac{2-\gamma}{2}}$ are $C^\infty$ in a neighborhood of $x_0$.
\end{itemize}
\end{theorem}

In addition to its own intrinsic interest, this result is important for example in the study of stable solutions to the Alt--Phillips problem (in which one needs at least second order expansions for solutions near $\partial\Omega$), especially when $\gamma < \frac23$; see \cite{KS24}.

Our approach to prove Theorem \ref{thm 1} is to find a PDE for the derivatives $u_i$ (and for their quotients $u_i/u_j$), to then use these equations to deduce that solutions and free boundaries are smooth.
This approach was introduced by De Silva and Savin in the context of the obstacle problem \cite{DS15}, and also in the Alt--Caffarelli problem; see \cite{DS15b}.

In the obstacle problem, they proved that if $v_2>0$ in $\Omega\cap B_1$, then
\[\left. \begin{array}{r}
\Delta v_1=\Delta v_2=0  \quad \textrm{in}\quad\Omega\cap B_1\\
v_1=v_2=0  \quad \textrm{on}\quad\partial\Omega\cap B_1\\
\partial\Omega  \quad \textrm{is}\quad C^{k,\alpha}\\
\end{array}\right\} \qquad \Longrightarrow \qquad 
\frac{v_1}{v_2} \in C^{k,\alpha}(\overline\Omega \cap B_{1/2}),\]
for any $k\geq1$ and $\alpha\in(0,1)$.
This, applied to the derivatives $u_i$, and using that the normal vector $\nu$ to $\partial\Omega$ can be written as
\[\nu^i = \frac{u_i}{|\nabla u|} = \frac{u_i/u_n}{\sqrt{\sum_j (u_j/u_n)^2}}\]
yields that
\begin{equation}\label{strategy-obstacle}
\partial\Omega \in C^{k,\alpha} \quad\Longrightarrow\quad  \frac{u_i}{u_n}\in C^{k,\alpha} \quad\Longrightarrow \quad\nu \in C^{k,\alpha} \quad\Longrightarrow \quad\partial\Omega \in C^{k+1,\alpha},\end{equation}
thus yielding the $C^\infty$ regularity by a bootstrap argument.

In case of the Alt-Caffarelli free boundary problem, the argument goes in the same spirit, but the PDE is completely different.
It turns out that the quotient $w:=u_i/u_n$ solves
\[\begin{split}
{\rm div}(u_n^2 \nabla w)=0 & \quad \textrm{in}\quad \Omega \cap B_1 \\
\partial_\nu w =0 &\quad \textrm{on}\quad \partial\Omega\cap B_1, 
\end{split}\]
and thus by Schauder estimates we get
\begin{equation}\label{strategy-bernoulli}
\partial\Omega \in C^{k,\alpha} \quad\Longrightarrow\quad  u_n\in C^{k-1,\alpha} \quad\Longrightarrow\quad  w\in C^{k,\alpha} \quad\Longrightarrow \quad\nu \in C^{k,\alpha} \quad\Longrightarrow \quad\partial\Omega \in C^{k+1,\alpha}.\end{equation}
This yields that $\partial\Omega$ is $C^\infty$, and thus applying now Schauder estimates (with Dirichlet conditions) to the solution $u$, we find that $u\in C^\infty$.

\subsection{The linearized equation}

In our setting, derivatives $u_i$ are not harmonic but solve the linearized equation
\[\Delta u_i = (\gamma-1)u^{\gamma-2} u_i \quad \textrm{in}\quad \Omega\cap B_1.\]
At regular points, we know that 
\[u \sim c_\gamma d^{\frac{2}{2-\gamma}}\]
where $d$ is the distance to the (free) boundary and\footnote{Intuitively, this is because blow-ups of solutions at regular points converge to the function $c_\gamma (x\cdot\nu)_+^\frac{2}{2-\gamma}$, the one dimensional solution of the Alt-Phillips equation.} 
\begin{equation}\label{eq coefficient distance}
	c_\gamma := \left(\frac{(2-\gamma)^2}{2 \gamma}\right)^\frac{1}{2-\gamma};
\end{equation}
see \cite{AP86}.
Thus, since near regular points we also know that $u/d^{\frac{2}{2-\gamma}} \in C^\alpha$, then the linearized equation becomes
\begin{equation}\label{asdf}
\Delta u_i = \kappa_\gamma \frac{u_i}{d^2} + f d^{\frac{\gamma}{2-\gamma}-2} \quad \textrm{in}\quad \Omega\cap B_1,
\end{equation}
where $f|_{\partial\Omega}=0$, $f\in C^\alpha(\overline\Omega\cap B_1)$, and 
 \begin{equation}\label{eq aux constant} 
 	\kappa_\gamma := (\gamma-1)c_\gamma^{\gamma-2}= \Big(\frac{\gamma}{2-\gamma}\Big) \Big(\frac{\gamma}{2-\gamma}-1\Big).
 \end{equation}
Notice that 
\[\kappa_\gamma \geq -{\textstyle\frac14},\]
with the minimum being attained at $\gamma=\frac23$.
Moreover, we also have the symmetry property
\[\kappa_\gamma = \kappa_{\frac{4-4\gamma}{4-3\gamma}}.\]

Thus, we define the linearized operator
\begin{equation}\label{eq linearized alt-phil 0}
	L_\kappa(v):=-\Delta v + \kappa \,\frac{v}{d^2},
\end{equation}
Quite interestingly, this operator is the appearing in the Euler-Lagrange equation of the Hardy-type inequality
\[\int_\Omega |\nabla \xi|^2 + \kappa \int_\Omega \frac{\xi^2}{d^2} \geq0,\]
which holds in bounded smooth domains for $\kappa>-\frac14$, and \textit{does not} hold for $\kappa<-\frac14$.
In case $\kappa= -
\frac14$ the inequality is known to hold in convex domains.
While we will not use such variational characterization, this shows that the range $\kappa\geq -\frac14$ is very natural, and the case $\kappa=- \frac14$ is critical.

Our analysis will be based on non-variational arguments, and in particular we will establish a comparison principle for $L_\kappa$ as well as suitable barriers.

As we will see, there are two characteristic powers for the operator in dimension one:
\[ L_\kappa (x^\theta) = 0 \quad \textrm{for}\quad x>0 \qquad \Longleftrightarrow \qquad \theta=\frac{1\pm\sqrt{1+4\kappa}}{2}.\]
We call the largest root $\theta_+$, whereas we will denote by $\theta_-$ the smallest root.
The first one will behave somehow as the exponent of solutions with Dirichlet conditions, while the second one somehow corresponds to the exponent of solutions with  Neumann conditions.
The case $\kappa=- \frac14$ is slightly different, as there is only one exponent $x^{1/2}$, but then there is a second 1D solution, namely $x^{1/2}\log x$.

Our main result for the linearized operator $L_\kappa$, which we believe is of independent interest, reads as follows.

\begin{theorem}\label{thm schauder}
Let $\kappa\in[-\frac14,\infty)$, and let $\Omega \subset \R^n$ be a $C^{k+1,\alpha}$ domain, $k\geq 0$ and $\alpha\in(0,1)$.
Let $f\in C^{k,\alpha}(  \overline\Omega \cap B_1 )$ be such that $f|_{\partial\Omega}\equiv0$, and $d$ be the regularized distance to $\partial\Omega$, given by Lemma \ref{lemma generalized distance}.

Let $v \in C(\Omega \cap B_1)$ be any solution to 
\begin{equation}\label{eq boundary linearized pde}
-\Delta v + \kappa\,\frac{v}{d^2}=fd^{\theta_+-2}\quad \text{in}\quad  \Omega \cap  B_1 ,
\end{equation}
satisfying
\begin{equation}\label{bdry-cond}
\begin{array}{rl}
\displaystyle \lim_{x\to \partial\Omega} \frac{v}{d^{\theta_-}} = 0 & \quad \textrm{if}\quad \kappa>-{\textstyle \frac14} \vspace{1mm} \\
\displaystyle \lim_{x\to \partial\Omega} \frac{v}{d^{1/2} |\log d|} = 0 & \quad \textrm{if}\quad \kappa=-{\textstyle \frac14} .
\end{array}
\end{equation}
Then, $v\in C^{\theta_+}( B_{1/2}\cap \overline\Omega)$ and 
\[
\big\|v/{d^{\theta_+}}\big\|_{C^{k,\alpha}(\overline\Omega \cap B_{1/2})}\leq C\big(\|f\|_{C^{k,\alpha}(\overline\Omega\cap B_1)} + \|v\|_{L^\infty(\Omega \cap B_1)} \big),
\]
where $C>0$ depends only on $\Omega$, $k$, $\alpha$, $n$, and $\kappa$.
\end{theorem}

Here, the assumption \eqref{bdry-cond} acts as a Dirichlet boundary condition on~$\partial\Omega$ ruling out the existence of solutions behaving like the other possible power. This is interesting because $v$ is not assumed to be bounded and since  $\theta_- <0$ when $\kappa >0$, \eqref{bdry-cond} imposes a growth bound nearby the boundary which, a posteriori, implies the boundedness of $v$.\\ 

We will also prove a modified version of the result allowing for general pointwise expansions for $v$ of the form $Q_1 d^{\theta_+}+ Q_2 d^{\theta_-+1}$ nearby $\partial\Omega$ with $Q_1$ and $Q_2$ polynomials of a suitable degree; see Theorem \ref{thm higher reg linear}. 

\subsection{Strategy of the proofs}

Recall that solutions $u$ to the Alt--Phillips problem behave like~$d^{\frac{2}{2-\gamma}}$, and thus their derivatives $u_i$ do so like $d^{\frac{\gamma}{2-\gamma}}$.
This means that we have two completely different cases: for $\gamma\geq\frac23$ the exponent $\frac{\gamma}{2-\gamma}$ corresponds to $\theta_+$ of the linearized operator~$L_\kappa$
\[\gamma\in[{\textstyle \frac23},2) \qquad \Longrightarrow \qquad u_i \asymp d^{\theta_+},
\]
while for $\gamma<\frac23$ the exponent $\frac{\gamma}{2-\gamma}$ corresponds to $\theta_-$ of  $L_\kappa$
\[\gamma\in(0,{\textstyle \frac23}) \qquad \Longrightarrow \qquad u_i \asymp d^{\theta_-}.
\]
In a sense, this means that \emph{when $\gamma\geq\frac23$ derivatives $u_i$ have zero Dirichlet boundary data}, while for $\gamma<\frac23$ this is not true anymore.

In both cases, the idea is to apply Theorem \ref{thm schauder}  inductively to show that $u_i/d^{\frac{\gamma}{2-\gamma}}$, $u/d^{\frac{2}{2-\gamma}}$, and  $u_i/u_n$ are $C^\infty$, namely,
\begin{equation}\label{strategy-proof}
\partial\Omega \in C^{k+1,\alpha} \quad \overset{(a)}{\Longrightarrow} \quad  \frac{u_i}{d^{\frac{\gamma}{2-\gamma}}}, \frac{u}{d^{\frac{2}{2-\gamma}}}\in C^{k,\alpha} \quad \overset{(b)}{\Longrightarrow} \quad  \frac{u_i}{u_n} \in C^{k+1,\alpha} \quad\Longrightarrow \quad\partial\Omega \in C^{k+2,\alpha}.\end{equation}

Let us notice that since nearby regular points  $u/d^{\frac{2}{2-\gamma}} \in C^\alpha$ (see Lemma \ref{prop regular solutions}), from the scheme \eqref{strategy-proof}, (a) only has to be considered when $k\geq 1$. However, it is not immediate to apply Theorem \ref{thm schauder} to the derivatives $u_i$ in order to get (a).
Indeed, by induction hypothesis we know that $u/d^{\frac{2}{2-\gamma}}\in C^{k-1,\alpha}$ and $u_i/d^{\frac{\gamma}{2-\gamma}}\in C^{k-1,\alpha}$, so that the right hand side $f$ in the equation \eqref{asdf} is then $C^{k-1,\alpha}$.
Unfortunately, using Theorem  \ref{thm schauder} (when $\gamma \geq \frac{2}{3}$) we then deduce that  $u_i/d^{\frac{\gamma}{2-\gamma}}\in C^{k-1,\alpha}$, gaining no new information.

To improve the smoothness of $u/d^{\frac{\gamma}{2-\gamma}}$ and $u_i/d^{\frac{\gamma}{2-\gamma}}$, a key idea is to consider the function 
\[v:=u-c_\gamma d^{\frac{2}{2-\gamma}},\] 
and use the identity
\[p^{\gamma-1}-q^{\gamma-1}=(\gamma-1)(p-q)q^{\gamma-2} + (p/q-1)^2 q^{\gamma-1} \int_0^1 (\gamma-1)(\gamma-2) \big(1+t p/q\big)^{\gamma-3}(1-t) dt \]
for $p,q>0$, to show that
\begin{equation}\label{eq linearized eq}
\begin{split}
	L_\kappa v & = u^{\gamma-1}-\big(c_\gamma d^{\frac{2}{2-\gamma}}\big)^{\gamma-1}+gd^{\frac{2}{2-\gamma}-2} \\
 & = \kappa \frac{v}{d^2} +  f d^{\frac{2}{2-\gamma}-2} + gd^{\frac{2}{2-\gamma}-2},
\end{split}
\end{equation}
where $g\in C^{k,\alpha}$ and
\[f:= \big(u/d^{\frac{2}{2-\gamma}}-c_\gamma\big)^2 \int_0^1 \bar c_\gamma\big(c_\gamma+t u/d^{\frac{2}{2-\gamma}}\big)^{\gamma-3}(1-t) dt.\]

The crucial point here is that, even though by inductive hypothesis we only know $u/d^{\frac{2}{2-\gamma}}$ is $C^{k-1,\alpha}$, we will show that $f$ is actually \emph{more} regular, namely $f\in C^{k,\alpha}$ if $k\geq2$.

The reason for this is that, roughly speaking, if a function $w$ is \textit{pointwise} $C^{m,\alpha}$  at $x=0$, and $w(0)=0$, then $w^2$ is actually pointwise $C^{m+1,\alpha}$ at $x=0$ if \footnote{When $m=0$, we have instead that $w^2$ is $C^{2\alpha}$ at $0$.} $m\geq1$.
And $w^2\eta$ will still be pointwise $C^{m+1,\alpha}$ as long as $\eta$ is $C^{m,\alpha}$.
This, yields that $f$ is pointwise $C^{k,\alpha}$ at all boundary points, and combined with interior estimates for $u$ allows us to prove that $f$ is indeed $C^{k,\alpha}$ in~$\overline\Omega$. Combining this with the fact that $v$ satisfies simultaneously \eqref{bdry-cond} and \eqref{bdry-cond higher}, it is possible to prove a modified version of Theorem \ref{thm schauder}, i.e., Theorem \ref{thm higher reg solutions}, to deduce $v/d^{\frac{2}{2-\gamma}} \in C^{k,\alpha}$.

Recalling the definition of $v$, and using again interior estimates for $u$, we then prove the desired estimate $u_i/d^{\frac{\gamma}{2-\gamma}} \in C^{k,\alpha}$ in (a).

\vspace{2mm}

Once we have (a) in \eqref{strategy-proof}, to prove (b) we notice that the quotient $w:=u_i/u_n$ satisfies
\[\begin{split}
{\rm div}(u_n^2 \nabla w)=0 & \quad \textrm{in}\quad \Omega \cap B_1 \\
\partial_\nu w =0 &\quad \textrm{on}\quad \partial\Omega\cap B_1, 
\end{split}\]
which is a degenerate elliptic equation since $u_n\to0$ on $\partial\Omega$.
Still, thanks to (a), we have $u_n^2=d^s a(x)$, with $s=\frac{2\gamma}{2-\gamma}$ and $a(x)\in C^{k,\alpha}$.
Using this, we can apply\footnote{Here, one has to be careful with the case $k=1$, in which the boundary condition $\partial_\nu w$ must be understood in an appropriate sense.} Schauder estimates for degenerate elliptic equations with Neumann boundary conditions due to Terraccini-Tortone-Vita~\cite{TTV22}, to show that $w\in C^{k+1,\alpha}$.

\vspace{2mm}

Although a key point in showing (a)  --for all values of $\gamma\in(0,2)$-- is Theorem \ref{thm higher reg solutions} instead of Theorem \ref{thm schauder}, the former is actually a consequence of the latter. Also,  Theorem \ref{thm schauder} has direct applications to the study of the boundary behavior of solutions to equations involving operators with Hardy-type potentials (see \cite{GV15}).\\

To prove Theorem \ref{thm schauder}, we first need a comparison principle for the operator $L_\kappa$ with boundary conditions \eqref{bdry-cond}.
We prove it by constructing explicit barriers, and reads as follows:

\begin{lemma}\label{lem-comparison}
Let $\kappa\in[-\frac14,\infty)$, $\alpha \in (0,1)$, $\Omega \subset \R^n$ be such that $\Omega \cap B_1$ is a $C^{1,\alpha}$ domain, and $d$ be as in Theorem \ref{thm schauder}.

Let $v \in C( \Omega \cap B_1)$ be any viscosity solution of
\[
-\Delta v + \kappa\,\frac{v}{d^2}\leq0\quad \text{in}\quad \Omega\cap B_1
\]
satisfying
\begin{equation}\label{bdry-cond2}
\begin{array}{rl}
\displaystyle \lim_{x\to \partial\Omega} \frac{v}{d^{\theta_-}} \leq 0 & \quad \textrm{if}\quad \kappa>-{\textstyle \frac14} \vspace{1mm} \\
\displaystyle \lim_{x\to \partial\Omega} \frac{v}{d^{1/2} |\log d|} \leq 0 & \quad \textrm{if}\quad \kappa=-{\textstyle \frac14} .
\end{array}
\end{equation}
Then, there exists $\rho \in (0 ,1)$ only depending on the $C^{1,\alpha}$ norm of $\partial \Omega$ such that for any $r\in (0,\rho]$ if  $v\leq0$ on $  \partial \Omega \cap B_r$, then  $v\leq0$ in $ \Omega \cap B_r$.
\end{lemma}

Using this, we prove uniform $C^{\theta_+}$ estimates for solutions to the linearized equation $L_\kappa v=0$ with zero Dirichlet boundary conditions.
This gives us compactness on the family of solutions $v$ we are studying, and thus we can run a contradiction and blow-up argument to get the higher order expansion, which is developed in Section \ref{sec4}.
This, combined with a classification of global solutions  in a half-space (see Proposition \ref{lemma general Liouville}), will yield Theorem \ref{thm schauder}, as well as its generalized version Theorem \ref{thm higher reg linear}.

\subsection{Acknowledgements}

X.R. was supported by the European Research Council (ERC) under the Grant Agreement No 101123223 (\textsc{SSNSD}), by the AEI project PID2021-125021NAI00 (Spain), AGAUR Grant 2021 SGR 00087 (Catalunya), AEI Grant RED2022-134784-T funded by MCIN/AEI/10.13039/501100011033 (Spain), and the AEI Maria de Maeztu Program for Centers and Units of Excellence in R\&D CEX2020-001084-M.

\subsection{Organization of the paper}

In Section \ref{sec2} we give some preliminary lemmas.
Then, in Sections \ref{sec3} and \ref{sec4} we develop the boundary regularity for the linearized operator $L_\kappa$, proving in particular Theorem \ref{thm schauder}.
Finally, in Section \ref{sec5} we prove Theorem \ref{thm 1}, while in Section \ref{sec6} we give some technical lemmas that are used throughout the proofs.

\subsection{Notations and conventions}

We denote $\mathbf{P}_{k}$ the space of polynomials of order $k$ in $n$ variables. Polynomials will be written in their coefficient expansion following convention
$$Q \in \mathbf{P}_{k} \qquad \implies \qquad Q(x)= \sum_{\alpha \in \N^n\,,\, |\alpha|\leq k} q^{(\alpha)}x^\alpha, \quad x\in \R^n.$$

In the study of boundary regularity in a domain $\Omega$, we will always assume that $0\in \partial \Omega$ and, whenever $\Omega$ is regular at $0$, we will orient it so that the canonical vector $-e_n$ corresponds to its outer unit normal at zero, unless otherwise stated. 
Note that these two simplifications are possible thanks to the translation and rotation invariances of \eqref{eq alt-phil}. 
In concordance with this assumption, we will adopt the notation $x= (x',x_n) \in \R^{n-1}\times \R$, calling the first $n-1$ coordinates ``tangential'', whereas we will refer to the coordinate $x_n$ as the normal direction.

Associated to $\kappa\in [\frac{-1}{4}, \infty)$, we define the quantities 
\begin{equation}\label{eq powers}
\theta_+ =\frac{1+\sqrt{1+4\kappa}}{2}, \qquad \theta_-= \frac{1-\sqrt{1+4\kappa}}{2},
\end{equation}
which arise in the study of the linearized operator \eqref{eq linearized alt-phil 0} as discussed in the introduction (see also Lemma \ref{lemma one dimensional problem}).

When $\beta \notin \N$ we use the single index notation $C^{\beta}$ for the H\"older space $C^{\lfloor\beta\rfloor, \beta-\lfloor \beta \rfloor}$ where $\lfloor \cdot \rfloor$ denotes the integer part of a positive real number. Regarding H\"older norms of domains, we will assume throughout this paper that $ B_2 \cap \Omega$ is a graph with $C^{\beta}$ norm bounded by one. We remark that this assumption does not incur in any loss of generality since any solution $u$ to \eqref{eq alt-phil} can be rescaled as $u(rx)/r^\frac{2}{2-\gamma}$ to obtain a new solution in a domain with smaller H\"older norm where the graphycality assumption holds. 

Finally, following a widely used notation, we will use $C$ within proofs of propositions to indicate a unspecified constant, whose value will be allowed to change from line to line, and whose dependences will be clearly established in the statements of the corresponding propositions.  We will make use of sub-indices whenever we will want to underline certain hidden dependencies of the constant.

\section{Regularized distance and first estimates}
\label{sec2}

In our analysis we will use some classical properties of the distance function to the boundary and some of its generalizations. It is a classical result that, for $\beta\geq 1$, nearby $C^\beta$ pieces of boundary, the distance function is $C^\beta$. A priori the domains we are dealing with are $C^{1,\alpha}$, this forces us to consider regularized versions of the distance function.\\

\begin{lemma}\label{lemma generalized distance}
	Let $\beta>1$ with $\beta\notin \N$, and let $\partial \Omega \cap B_2$ be a $C^{\beta}$ graph. Then, there exist a generalized distance function $d \in C^{\infty}( \Omega \cap B_1)\cap C^{\beta}(\overline{\Omega}\cap B_1)$ satisfying
	\begin{enumerate}
		
		\item 
		\begin{equation}\label{eq equivalence distance lieberman}
			\frac{1}{C} \text{dist}(x)\leq d(x) \leq  	C \text{dist}(x), \qquad x\in \Omega \cap B_1,
		\end{equation}
		where $ \text{dist}(x):= \text{dist}(x, \partial \Omega)$.
		
		\item 
	\begin{equation}\label{eq expansion distance}
		d(x)= x_n +Q(x)+g(x), \qquad x\in  \Omega\cap B_1,
	\end{equation}
		where $Q \in \mathbf{P}_{\lfloor \beta \rfloor}$ with its zeroth and first order terms identically equal to zero and also satisfying $\Vert Q \Vert_{L^\infty(B_1)}\leq C$, and where $g \in C^{\infty}(\Omega \cap B_1)$ satisfies $|\nabla g| \leq C|x|^{\beta-1}$.

		\item For any $\lambda \in \R$, if $\beta \in (1,2)$
		
		\begin{equation}\label{eq derivative power}
		|\Delta d^\lambda - \lambda(\lambda-1) d^{\lambda-2}|\leq C d^{\beta+\lambda-3} \, \mbox{ in  $ \Omega\cap B_1$,}
		\end{equation}
		and if $\beta >2$, then
		
		\begin{equation}\label{eq lap distpower}
			\Delta d^\lambda - \lambda(\lambda-1) d^{\lambda-2}= d^{\lambda-1}g(x), \quad \mbox{in  $ \Omega\cap B_1$,}
		\end{equation}
	with $g \in C^{\beta-2}( \overline{\Omega}\cap B_1)\cap  C^{\infty}(  \Omega\cap B_1)$.
		 
		\item 
		In particular, we have that
		\begin{equation}\label{eq laplacian altphillips}
		\Delta \Big( c_\gamma d^\frac{2}{2-\gamma} \Big) -  \Big( c_\gamma d^\frac{2}{2-\gamma} \Big)^{\gamma-1} = h \, \mbox{ in  $ \Omega\cap B_1$,}
		\end{equation}		
with $|h|\leq Cd^{\beta +\frac{2}{2-\gamma}-3}$ if $\beta \in (1,2)$ and $h= d^\frac{\gamma}{2-\gamma}g(x)$ with $g \in C^{\beta-2}( \overline{\Omega}\cap B_1)$ if $\beta>2$.

	\item If $\beta \in (1,2)$
	 \begin{eqnarray}\label{eq crit barrier}
		-\Delta \Big(d^\frac{1}{2} |\ln(d)|^\frac{1}{2}\Big)+ \frac{1}{4} \frac{d^\frac{2}{3} |\ln(d)|^\frac{1}{2}}{d^2} = \frac{1 }{4d^{\frac{3}{2}}|\ln(d)|^\frac{3}{2}}\Big(1 +h\Big),
	\end{eqnarray}
 in  $ \cap \{d<1\}\cap B_1$ with $|h|\leq C(\ln(d))^2d^{\beta-1}$.
	\end{enumerate}
\end{lemma}
\begin{proof}
 Let us start addressing (1) and (2). We follow the construction in \cite{L} (see also \cite[Lemma B.0.1]{FR24} considering
\begin{equation}\label{eq lieberman conv}
	G(x,\rho) = \int_{B_1} \text{dist}\Big(x-\frac{\rho}{2}z\Big)\phi(z)dz
\end{equation}
with $\phi \in C_c^\infty(B_1)$ a radial smooth non-negative function satisfying $\int_{B_1} \phi =1$. Then $d$ is defined implicitly by $d(x) = G(x,d(x))$. In virtue of the implicit function theorem it follows that   $d \in C^{\beta}( \overline{\Omega}\cap B_1)\cap C^{\infty}(\Omega \cap B_1)$. Additionally, (1) is a consequence of the construction as discussed in \cite{L}. On the other hand, by differentiating the defining expression for $d$ at a point $x\in  \partial \Omega \cap B_2$, we get
\begin{equation}\label{eq gradient on the boundary}
	\nabla d(x) = \nabla \text{dist}(x) \int_{B_1} \phi(z)dz - \frac{1}{2}\nabla \text{dist}(x) \cdot \int_{B_1} z\phi(z)dz= \nabla \text{dist}(x),
\end{equation}
thanks to the properties of $\phi$. Additionally, thanks to the uniform estimates on $d$ and the orientation assumption on $\Omega$ (2) follows. We also remark that this construction satisfies
\begin{equation}\label{eq der distance}
	|\partial^{\mu} d|\leq C_{k} d^{\beta-|\mu|}, \mbox{ for $|\mu|>\beta$}.
\end{equation}
	
For part (3), a direct computation yields
\begin{equation}\label{eq laplace powerdist}
	\Delta d^\lambda = \lambda(\lambda-1)|\nabla d|^2 \frac{d^\lambda}{d^2}+\lambda d^{\lambda-1}\Delta d.
\end{equation}

So, if $\beta \in (1,2)$, we deduce from \eqref{eq gradient on the boundary} that $|\nabla d(x)| =1$ on $ \partial \Omega\cap B_1$, which combined with the $C^\beta$ regularity implies that
\begin{equation}\label{eq expansion gradsquare distance}
	\Big||\nabla d|^2 - 1\big| \leq C d^{\beta-1}\, \, \mbox{ in  $ \Omega \cap B_1$}.
\end{equation}

So, by combining \eqref{eq der distance}, \eqref{eq laplace powerdist}, and \eqref{eq expansion gradsquare distance} we deduce \eqref{eq derivative power}. When $\beta>2$, we have that the usual distance function satisfies \eqref{eq expansion distance} (see, e.g., \cite[Section 14.6]{GT}), which is inherited by $d$. Hence, these considerations altogether with \eqref{eq coefficient distance} imply (4), (4) realarly, we argue by a direct computation to show part (5). Given $\eta>0$, from \eqref{eq derivative power} and \eqref{eq der distance} we deduce
\begin{eqnarray*}
	\Delta d^\lambda |\ln(d)|^\eta &=& \Delta(d^\lambda) |\ln(d)|^\eta+2 \nabla (d^\lambda)\cdot \nabla  |\ln(d)|^\eta+d^\lambda \Delta( |\ln(d)|^\eta)\\
	&=& \Big(\lambda(\lambda-1) \frac{d^\lambda |\ln(d)|^\eta}{d^2}+ \Big(\frac{(\eta-1)}{|\ln(d)|}-2\lambda+1\Big)\frac{d^\lambda \eta |\ln(d)|^{\eta-1}}{d^2}\Big)|\nabla d|^2\\
	&&+ |\ln(d)|^\eta O( d^{\beta+\lambda-3})+|\ln(d)|^{\eta-1} O( d^{\beta+\lambda-3})
\end{eqnarray*}

In particular, if $\lambda= \frac{1}{2}$ and $\eta=1/2$, the previous expression together with \eqref{eq expansion gradsquare distance} yields the expansion
\begin{equation*}
		\Delta \Big( d^\frac{1}{2} |\ln(d)|^\frac{1}{2}\Big) =\frac{1}{4} \frac{d^\frac{2}{3} |\ln(d)|^\frac{1}{2}}{d^2}-\frac{1 }{4d^{\frac{3}{2}}|\ln(d)|^\frac{3}{2}}\Big(1 +O(\ln(d)^2d^{\beta-1})\Big),
\end{equation*}
which concludes the proof.
\end{proof}

In a similar vein, we construct the following barrier function that is still a sort of generalized distance satisfying other suitable properties.
\begin{lemma}\label{lemma distance barrier}
	Let $\rho \in (0,\frac{1}{2})$, $\beta \in (1,2)$ and let $\partial \Omega \cap B_1$ be a $C^{\beta}$ graph.  Then, there exist a constant $C>0$ only depending on $\beta$ and $\rho$, and a function $\phi \in C^{\infty}(\Omega \cap B_1)\cap C^{\beta}( \overline{\Omega}\cap B_r) $ for any $r\in (0,1)$ such that
	\begin{enumerate}
	\item 	
\begin{equation}\label{eq equivalence phi}
	\frac{1}{C} \text{d}(x)\leq \phi(x) \leq  	C \text{d}(x), \qquad x\in  \Omega\cap B_\rho,
\end{equation}
\item
\begin{equation}\label{eq boundary distance phi}
	\phi(x)\geq \frac{1}{C}\quad \mbox{for $x \in \Omega\cap  B_{2\rho}^c$},
\end{equation}			
\item 	 
\begin{equation}\label{eq expansion phi}
	\phi(x)= x_n +g(x), \qquad x\in \Omega\cap B_1,
\end{equation}
where $g \in C^{\infty}(\Omega \cap B_2)$ satisfies $|\nabla g| \leq C|x|^{\beta-1}$ in $\Omega\cap B_1$,

	\item and for $\lambda >0$ we have that			
	\begin{equation}\label{eq derivative power phi}
		|\Delta \phi^\lambda - \lambda(\lambda-1) \phi^{\lambda-2}|\leq C |x|^{\beta+\lambda-3} \, \mbox{ in  $\Omega \cap B_1$.}
	\end{equation}
\end{enumerate}
\end{lemma}
	
\begin{proof}
	Let $\phi_0$ be the unique solution of
	\begin{equation}\label{eq gen distance}
		\begin{cases}
			\Delta \phi_0 = 0,  & \Omega\cap B_1,\\
			\phi_0=\text{dist}(\cdot,  B_{\rho}\cap \partial \Omega  ),  &\partial \big(\Omega\cap B_1 \big),
		\end{cases}
	\end{equation}
	where $\text{dist}$ is the standard distance function to a set. Let us notice that $\phi_0\in C^{\infty}( \Omega\cap B_1)\cap C^{\beta}(\overline{ \Omega}\cap B_r)$ for any $r \in (0,1)$. By the maximum principle $\phi_0>0$ in $\Omega \cap B_1$. Thus, since $\phi_0$ vanishes tangentially along $\partial \Omega$, Hopf's lemma together with the assumption $\nu_{\Omega}(0)= -e_n$ imply that $\nabla \phi(0) = C_{\Omega} e_n$ for some $C_{\Omega}$ satisfying $\frac{1}{C(\beta)}\leq  C_{\Omega} \leq C(\beta)$ thanks to the uniform $C^{\beta}$ bound on $\Omega \cap B_1$. Putting $\phi= \frac{1}{C_\Omega}\phi_0$, we deduce \eqref{eq boundary distance phi} from standard regularity for harmonic functions \cite[Lemma A.2]{AR}, whereas \eqref{eq boundary distance phi} follows directly from the definition of $\phi_0$.\\
	
	Part (3)  follows from boundary Schauder estimates and the uniform bound on the $C^{\beta}$ norm of the domain. Lastly, for part (4) we argue as in Lemma \ref{lemma generalized distance}
	\begin{equation}\label{eq laplace phi}
		\Delta \phi^\lambda = \lambda(\lambda-1)|\nabla \phi|^2 \frac{d^\lambda}{d^2},
	\end{equation}
	where we have used the harmonicity of $\phi$. So, by combining \eqref{eq laplace phi}, with the fact that $|\nabla \phi(0)|=1$ and the expanssion \eqref{eq expansion phi}, we obtain \eqref{eq derivative power phi}.	
\end{proof}

In the following lemma, we recall some properties of solutions to \eqref{eq alt-phil} nearby regular points.

\begin{proposition}\label{prop regular solutions}
	Let $u$ be a solution to \eqref{eq alt-phil} and let $0 \in \partial \Omega$ be a regular free boundary point. Then, there exists $r_0>0$ such that $\partial \Omega$ is $C^{1,\alpha}$ for any $\alpha \in (0,1)$. Moreover, there exist $ C, \rho>0$ (depending only on $\alpha, \gamma$, and $n$) such that 
	\begin{equation}\label{eq basic regularity dist}
		\Big|u(x)-c_\gamma d^\frac{2}{2-\gamma}(x)\Big| \leq C |x|^{{\frac{2}{2-\gamma} + \alpha}} \quad \mbox{for $x\in \Omega\cap B_{\rho}$}.
	\end{equation}	
	
	Additionally, 
	\begin{equation}\label{eq normal bhevior}
	 d(x)^{{\frac{\gamma}{2-\gamma}}}\leq Cu_n(x) \quad \mbox{for $x\in \Omega\cap B_{\rho}$}.
	\end{equation}
\end{proposition}
\begin{proof}
	Let us start recalling that nearby $0$, $\partial \Omega$ is $C^{1,\alpha_0}$ for some $\alpha_0$ in virtue of the results in  \cite{AP86, DS21}. Our goal is to show that $\alpha_0$ can be taken arbitrarily close to 1. We will show that this can be inferred from  Appendix B in \cite{FY23} - whose results, in that appendix, even if stated for $\gamma \in (0,1)$, hold true for the full range $\gamma \in (0,2)$ with exactly the same proofs.\\
	
	Let us start deriving \eqref{eq basic regularity dist}. Set $w:= \sqrt{\frac{\gamma}{2}}\beta v^\frac{2-\gamma}{2}$, which is normalized so that $|\nabla w(0)|=1$, see \cite{DS21}. Also, let us extend $w$ as zero outside of $\Omega$. Since $0$ is a regular point for $u$ and $\nu(0)=-e_n$, given any $\e>0$, we can find $\rho_\e$ such that 
	\begin{equation}\label{eq initial flatness}
		(x_n - C\rho_\e\e)_+ \leq w(x) \leq (x_n + C\rho_\e \e)_+, \quad \mbox{ $x\in B_{r}$}
	\end{equation}
	where $C>0$ is a universal constant.  Let $\alpha \in (0,1)$, in virtue of \cite[Proposition B.2]{FY23}, if we select $\e_0>0$ (and in consequence $\rho_0$) small enough in \eqref{eq initial flatness}, we have that for any $r \in (0, \rho_0/2]$ 
	\begin{equation*}\label{eq initial step}
		|w(x) - (x\cdot e_0)_+ |\leq C\e_0 r^{1+\alpha}, \quad \mbox{ $x\in B_{r}$}
	\end{equation*}
	where $C=C(\alpha, \gamma)$ and for some $e_0 \in \mathbb{S}^{n-1}$. Equivalently, we have that 
	\begin{equation*}
		|w_r(x) - (x\cdot e_0)_+ |\leq C\e r^{\alpha}, \quad \mbox{ $x\in B_{1}$}
	\end{equation*}
	where $w_r(x)= \frac{w(r x)}{r}$. However, since $0$ is a regular point for $u$ and  $\nu(0)=-e_n$, we have that $w_r(x)\to (x_n)_+$ as $r\to 0^+$, implying that $e_0=e_n$ and therefore
	\begin{equation*}
		|w_r(x) - (x_n)_+ |\leq C\e_0 r^{\alpha} \quad \mbox{ $x\in B_{1}$}
	\end{equation*}
	for $r \in (0, \rho_0/2)$. From this last inequality we easily deduce
	\begin{equation}\label{eq approximation}
		|w(x) - (x_n )_+ |\leq C|x|^{1+\alpha} \quad \mbox{ $x\in B_{\rho_0/2}$.}
	\end{equation}
	
	Repeating the same argument in a neighborhood of $0$, we find $\rho_1>0$ such that for $z\in B_{\rho_1}\cap \partial \Omega$
	\begin{equation}\label{eq app anypoint}
		\big|w(z+x) - (x \cdot (-\nu(z)) )_+ \big|\leq C|x|^{1+\alpha} \quad \mbox{ $x\in B_{\rho_1}(z)$,}
	\end{equation}
	where $\nu(z)$ is the outer normal of $\partial \Omega$ at $z$.  Thus, by applying \eqref{eq approximation} at the value $z\in B_{\rho_1/2}\cap \partial \Omega$ and \eqref{eq app anypoint} for this same $z$ at the value $x-z$ for some $x  \in B_{\rho_1/2}$, we deduce
	\begin{eqnarray}\notag
		\big| (x_n )_+-((x-z) \cdot (-\nu(z)) )_+ \big|&\leq& 	|(x_n )_+-w(x)|+|w(x) - ( (x-z) \cdot (-\nu(z)) )_+ |\\ \label{eq hold gradient}
		&\leq& C\Big(|x|^{1+\alpha}+|z-x|^{1+\alpha}\Big).
	\end{eqnarray}
	
	On top of this, by plugging $z$ in \eqref{eq approximation}, we have that $|z \cdot e_n|= |(z_n)_+|\leq C|z|^{1+\alpha}$ which combined with \eqref{eq hold gradient} yields
	\begin{eqnarray}\label{eq hold gradient 2}
		\big|(x-z)\cdot e_n -((x-z) \cdot (-\nu(z))_+ \big| \leq C\Big(|x|^{1+\alpha}+|z-x|^{1+\alpha}+|z|^{1+\alpha}\Big).
	\end{eqnarray}
	
	On the other hand, in virtue of the flatness property \eqref{eq initial flatness} we have that 
	\begin{equation}\label{eq choosing x}
		\{x \in  B_{\rho_1} | x_n\geq C\e \rho_1 \} \subset \{w> 0\} \quad \mbox{and} \quad |z_n|\leq C\e |z|.
	\end{equation}
	Thus, upon taking $\e_0>0$ small enough in \eqref{eq initial flatness} (and thus shrinking further $\rho_1$), given $y\in \mathbb{S}^{n-1}$ such that $y_n \geq \frac{1}{4}$ we can find $\{x \in  B_{\rho_1} | x_n\geq C\e \rho_1 \}$ satisfying  $|x|\leq C_1 |z|\leq C_2|x|$ such that $y=\frac{x-z}{|x-z|}$. We observe now that in virtue of \eqref{eq choosing x} and since the norms of $|x|$ and $|z|$ are comparable, we have that $|x|\leq C_1 |z-x|\leq C_2|x|$. Additionally, thanks to the  $C^{1,\alpha_0}$ of $\partial \Omega$, upon further shrinking $\rho_1$, we that $(x-z)\cdot  (-\nu(z))>0$. Thus, by incorporating these considerations into \eqref{eq hold gradient 2}, we deduce
	\begin{eqnarray*}
		\left|\frac{(x-z)}{|x-z|}\cdot (e_n -\nu(z)) \right| \leq C|z|^{\alpha},
	\end{eqnarray*}
	which in turn implies from our previous considerations
	\begin{eqnarray}\label{eq holder gradient}
		C\big| (e_n -\nu(z)) \big| = \sup_{ y \in \mathbb{S}^{n-1}, y_n \geq \frac{1}{4}} 	\big|y\cdot (e_n -\nu(z)) \big|  \leq C|z|^{\alpha}.
	\end{eqnarray}
	From here we deduce that $\partial \Omega$ is $C^{1,\alpha}$ in $B_{\rho_1}$. \\
	
	Aiming for \eqref{eq basic regularity dist}, we can exploit the $C^{1,\alpha}$ regularity of $\partial \Omega \cap B_{\rho_1}$ to deduce from Lemma \ref{lemma generalized distance} that $d \in C^{1,\alpha}(\Omega \cap B_{\rho_1})$. On top of this, since $\nabla d(0) =e_n$, we have that $|d(x)-(x_n)_+|\leq C|x|^{1+\alpha}$ for $x\in B_{\rho_1}$. This combined with \eqref{eq approximation} implies
	\begin{equation*}
		|w(x) - d(x)|\leq C|x|^{1+\alpha} \quad \mbox{ $x\in \Omega\cap B_{\rho_1}$.}
	\end{equation*}
	
	From here, we can apply the mean value theorem to the function $t\to t^\frac{2}{2-\gamma}$ and recall the definition of $c_\gamma$ in \eqref{eq coefficient distance} to deduce for $x\in B_{\rho_1}$
	\begin{eqnarray*}
		\left|u(x) - c_\gamma d^\frac{2}{2-\gamma}(x)\right| &=& C|w(x) - d(x)|\int_0^1 \big(tw(x)+(1-t)d(x)\big)^\frac{\gamma}{2-\gamma}\leq C|x|^{\frac{2}{2-\gamma}+\alpha}
	\end{eqnarray*}
	where we have used that $u(x)\leq Cd(x)^\frac{2}{2-\gamma}$ which, in turn, is consequence of \cite[Corollary 1.11]{AP86} combined with the boundary regularity of $\partial \Omega$ nearby 0.\\
	
	Finally, since $w_n(0) =1$, we deduce by continuity $\nabla w$ that $w_n \geq \frac{1}{C} u^\frac{1-\beta}{\beta} \geq  \frac{1}{C} d^\frac{\gamma}{2-\gamma}$ on $\Omega\cap B_{\rho}$ for $\rho$ small enough. This shows \eqref{eq normal bhevior} and finishes the proof.
\end{proof}

\section{Boundary regularity for solutions of the linearized equation $L_\kappa$.}
\label{sec3}

As indicated in the introduction, our analysis relies on understanding the one dimensional solutions of \eqref{eq linearized eq}. That is the purpose of the following lemma.

\begin{lemma}\label{lemma one dimensional problem}
		Let $\kappa \geq -\frac{1}{4}$ and $\theta_\pm$ defined as in \eqref{eq powers}. Given $c_i \in \R$ and $n_i \in \N \cup \{0\}$ for $i=1,2$, with $c_2 =0$ if $\theta_-\leq 0$. Let $y$ be  a solution of 
	\begin{equation}\label{eq one dimensional}
	y''+ \frac{\kappa}{x^2}y=c_1 x^{\theta_++n_1-1}+c_2 x^{\theta_-+n_2-1}\, \mbox{in $(0,\infty)$},
	\end{equation}
	satisfying 
	\begin{equation}\label{1d bdry-cond}
		\begin{array}{rl}
			\displaystyle \lim_{x\to 0^+} \frac{y}{x^{\theta_-}} = 0 & \quad \textrm{if}\quad \kappa>-{\textstyle \frac14} \vspace{1mm} \\
			\displaystyle \lim_{x\to 0^+} \frac{y}{x^{1/2} |\log x|} = 0 & \quad \textrm{if}\quad \kappa=-{\textstyle \frac14} .
		\end{array}
	\end{equation}

Then	
	\begin{equation}
	y(x)= \lambda_1x^{\theta_+}+Ac_1 x^{\theta_++n_1+1}+B c_2 x^{\theta_-+n_2+1}
	\end{equation}
where $ \lambda_1 \in \R$ is any number, and $A$ and $B$ are real valued parameters fully determined by $\kappa, n_1$ and $n_2$.	
\end{lemma}
\begin{proof}
	The homogeneous ODE associated with \eqref{eq one dimensional} is a standard Cauchy-Euler equation whose solutions are of the form $x^\theta$, with $\theta$ satisfying	
	\begin{equation*}
	\theta(\theta-1)+\kappa=0.
	\end{equation*}
	
	Therefore,	characteristic roots correspond to 
	$$\theta_\pm = \frac{1\pm \sqrt{1+4 \kappa}}{2},$$
	if $\kappa >\frac{-1}{4}$. However, $x^{\theta_-}$ is discarded by the boundary condition \eqref{1d bdry-cond}. On the other hand, if $\kappa = -\frac{1}{4}$, we have multiplicity and thus we must replace $x^{\theta_-}$ by $\ln(x)x^\frac{1}{2}$. This latter solution is also ruled out by \eqref{1d bdry-cond}. Finally, the structure of the particular solutions follows from standard computations in linear ODE theory, so we skip them for the sake of brevity.	
\end{proof}

\subsection{Boundary estimates}

Our goal now is to prove Lemma \ref{lem-comparison} and derive uniform boundary estimates for solutions from it. We start showing an elementary form of the maximum principle for sub/super solutions of $L_\kappa$ with suitable boundary behavior.

\begin{lemma}\label{eq comparison principle}
Let $D \subset \R^n$ be any bounded domain, and let $v, \psi \in C^2(D)$ satisfying the following properties
	\begin{eqnarray}\label{eq supersol}
	L_\kappa v&\leq& 0, \qquad  \mbox{in $D$},\\ \label{eq barrierphi}
	L_\kappa \psi&>&0, \qquad  \mbox{in $D$},\\\notag
	\psi&>&0, \qquad \mbox{in $D$},
\end{eqnarray}
	and such that
	\begin{equation}\label{eq vanishingatboundary}
     \limsup_{x \to \partial D}	\frac{v(x)}{\psi(x)} \leq 0.
	\end{equation}
	
	Then,  $v \leq 0$   in $D$.	
\end{lemma}
\begin{proof}
		 Let us notice that by \eqref{eq vanishingatboundary} either $v \leq 0$ in $D$ or  $\frac{v}{\psi}$ attains a positive maximum inside $D$. Arguing by contradiction, we assume that the latter holds. This means that there is $M>0$ and $z\in D$ such that 
		\begin{eqnarray*}
	v(z) &=& M \psi(z),\\
	v &\leq& M\psi, \quad \mbox{in D}.
	\end{eqnarray*}
	
	However, this latter observation yields $-\Delta v(z) \geq - M \Delta \psi(z)$, so we deduce that
	\begin{equation*}
			0\geq 	L_\kappa v(z) \geq M	L_\kappa \psi(z) > 0,
	\end{equation*} 
	a contradiction.
\end{proof}

As a corollary, we derive Lemma \ref{lem-comparison}.

\begin{proof}[Proof Lemma \ref{lem-comparison}]
Let $\beta = 1+\alpha$ so that $\Omega\cap B_1$ is a $C^{\beta}$ domain with $\beta \in (1,2)$. Let $\rho>0$ such that $d<1$ on $ \Omega\cap B_{\rho}$ and consider the function $\phi_1 = d^\frac{1}{2} (-\ln(d))^\frac{1}{2}$, which, thanks to \eqref{eq crit barrier}, satisfies
\begin{eqnarray}\label{eq barrier}
L_\kappa \phi_1 &=& \Big(\frac{1}{4}+\kappa\Big)\frac{ \phi_1}{d^2}+\frac{\phi_1 }{4d_1^{2}(-\ln(d_1))^2}\Big(1 +O(\ln(d)^2d^{\beta-1})\Big).
\end{eqnarray}

Since $d \leq C\rho$ in $B_{\rho}\cap \Omega$, we can choose $\rho$ such that $1 +O(\ln(\rho)^2\rho^{\min\{\beta-1,1\}}) \geq \frac{1}{2}$ in $\Omega \cap B_{\rho}$, depending only on the $C^{\beta}$ norm of $\Omega$. Hence, from \eqref{eq barrier} we deduce
\begin{equation}\label{eq barrier part 1}
	L_\kappa \phi_1 \geq \Big(\frac{1}{4}+\kappa \Big)\frac{ \phi_1}{d^2} +\frac{\phi_1}{d^2(-\ln(d_1))^2}(C-\rho^{\beta-1}(-\ln(\rho))^2)>0 \quad \mbox{ in $ \Omega \cap B_{\rho}$,}
\end{equation}
by shrinking $\rho$ further (if necessary) to make the last term positive. On the other hand, if $\kappa>-\frac{1}{4}$ (in which case we have that $\theta_- < \frac{1}{2}$), we can combine \eqref{eq derivative power} with the identity $\theta_-(\theta_- -1) =\kappa$ to derive the bounds
\begin{equation}\label{eq operator lower power}
	-C d^{\theta_--2}\rho^{\beta-1}\leq L_\kappa (d^{\theta_-})\leq C d^{\theta_--2}\rho^{\beta-1} \quad \mbox{ in $ \Omega \cap B_{\rho}$}.
\end{equation}

Regardless of the value of $\kappa$, we can combine \eqref{eq barrier part 1} and \eqref{eq operator lower power} to deduce that the barrier $\phi_2 = \phi_1+(1/4+\kappa)d^{\theta_-}$ satisfies 
\begin{equation*}
	L_\kappa(\phi_2) > \big(d^{1/2-\theta_-} (-\ln(d))^\frac{1}{2}+O(\rho^{\beta-1}) \big)\frac{ (1/4+\kappa) d^{\theta_-}}{d^2}>0 \quad \mbox{ in $\Omega\cap B_{\rho}$.}
\end{equation*}
Hence, if for some $r\leq \rho$, we have that $v\leq 0$ on $\Omega \cap \partial B_r $, we deduce that
\begin{equation*}
	\limsup_{x \to  \partial(\Omega \cap B_r)}	\frac{v(x)}{\phi_2(x)}\leq  0,
\end{equation*} 
in virtue of the boundary condition \eqref{bdry-cond2}. Finally, the result follows directly from Lemma \ref{eq supersol}.
\end{proof}

We apply now the previous lemma to prove a first step towards Theorem \ref{thm schauder}, which will be improved later on. In virtue of this lemma, we will derive  boundary estimates that will endow us with the required compactness to carry out our blow-up arguments. 

\begin{lemma}\label{lemma basic-boundary regularity}
	Let $\kappa\in[-\frac14,\infty)$, let $\beta \in (1,2)$ and let $\Omega \subset \R^n$ be a $C^{\beta}$ domain.
	 Let $d$ be the regularized distance to $\partial\Omega$, given by Lemma \ref{lemma generalized distance}.	
	 
	Let $v \in C(\Omega\cap B_1)$ be any solution to 
	\begin{equation}\label{eq boundary linearized pde statement}
		-\Delta v + \kappa\,\frac{v}{d^2}=f \quad \text{in}\quad \Omega\cap B_1,
	\end{equation}
	satisfying 
	\begin{equation}\label{bdry-cond statement}
		\begin{array}{rl}
			\displaystyle \lim_{x\to \partial\Omega} \frac{v}{d^{\theta_-}} = 0 & \quad \textrm{if}\quad \kappa>-{\textstyle \frac14} \vspace{1mm} \\
			\displaystyle \lim_{x\to \partial\Omega} \frac{v}{d^{1/2} |\log d|} = 0 & \quad \textrm{if}\quad \kappa=-{\textstyle \frac14},
		\end{array}
	\end{equation}
	where $f \in C(\Omega \cap B_1)$ satisfies
	\begin{equation}\label{eq boundary behavior f}
			f d^{2-\theta_+ -\e} \in L^\infty(\Omega \cap B_1 )
	\end{equation} 
for some $\varepsilon>0$.\\
	Then, there exists $\rho \in (0,1)$ such that
		\begin{equation}\label{eq boundary estimate}
			|v| \leq C\Big(\Vert v\Vert_{L^\infty(\Omega \cap B_1)}+ \Vert f \phi^{2-\theta_+ -\e} \Vert_{L^\infty(\Omega \cap B_1)}\Big)d^{\theta_+},\,\, \mbox{  in $\Omega \cap B_\rho$,}
\end{equation}
 where $\phi$ is the barrier constructed in Lemma \ref{lemma distance barrier} and $\rho$ depends only on $\beta$, $\varepsilon$ and $n$.\\
	
	Furthermore, if $\kappa <0$, then $v\in C^{\theta_+}(\overline\Omega\cap B_{1/2})$ and 
	\begin{equation}\label{eq holder estimate} 
	[v]_{\theta^+, \Omega \cap B_\rho }\leq C\Big(\Vert v\Vert_{L^\infty(\Omega \cap B_1)}+ \Vert  f \phi^{2-\theta_+ -\e}\Vert_{L^\infty(\Omega \cap B_1)}\Big).
\end{equation} 

Whilst, if $\kappa \geq 0$, we have the Lipschitz estimate 
	\begin{equation}\label{eq lip estimate} 
	[v]_{1, \Omega \cap B_\rho }\leq C\Big(\Vert v\Vert_{L^\infty(\Omega \cap B_1)}+ \Vert  f \phi^{2-\theta_+ -\e}\Vert_{L^\infty(\Omega \cap B_1)}\Big).
\end{equation} 
\end{lemma}

\begin{proof}
	Let $\rho \in (0,1/2)$ to be determined later in the proof and let $\phi$ be the barrier constructed in Lemma \ref{lemma distance barrier}. Thanks to \eqref{eq derivative power phi}, we can use the identity $(\theta_+-1)\theta_+ =\kappa$ as in \eqref{eq operator lower power} to deduce
\begin{equation}\label{eq operator higher power}
	|L_\kappa (\phi^{\theta_+})|\leq C \phi^{\theta_+-2}\rho^{\beta-1}
\end{equation}
in $B_{2\rho} \cap \Omega$.\\

Let $\e \in (0,\beta-1)$ such that \eqref{eq boundary behavior f} holds. Since $(\theta_++\e-1)(\theta_++\e)= \kappa+\e(2\theta_+-1+\e)$ and $\theta_+ \geq \frac{1}{2}$, we can proceed as in \eqref{eq operator higher power} to get
\begin{equation}\label{eq extra power barrier}
 \Big|	L_\kappa (\phi^{\theta_++\e})+\e(2\theta_+-1+\e) \phi^{\theta_++\e-2}\Big| \leq C\phi^{\theta_++\e-2}\rho^{\beta-1},
\end{equation}
$B_{2\rho} \cap \Omega$. \\

So, we can fix $\rho >0$ small enough depending on $\e$ and $\beta$ such that for some $\eta>0$
\begin{equation}\label{eq sign barrier}
	L_k \Big(\phi^{\theta_++\e}- \phi^{\theta_+}\Big) \leq -\eta \phi^{\theta^++\e-2}
\end{equation}
in $B_{2\rho} \cap \Omega$. Thus, after taking $\phi$ subordinated to $\rho$, we also have from \eqref{eq boundary distance phi} that
\begin{equation}\label{eq lower bound boundary ball}
\phi^{\theta_++\e}- \phi^{\theta_+}\geq \eta \mbox{ on $\Omega \cap \partial B_{2\rho}$}.
\end{equation}

Let us divide both sides of \eqref{eq boundary linearized pde statement} by 
$$A=\frac{\Vert v \Vert_{L^\infty(\Omega \cap B_1)}+\Vert \phi^{2-\theta_+ -\e} f \Vert_{L^\infty(\Omega \cap B_1)}}{\eta},$$
and let us notice that, by linearity, $v_1 =\frac{1}{A} v$ solves \eqref{eq boundary linearized pde statement} with right hand side $f_1 =\frac{1}{A}f$ and with the same boundary condition \eqref{bdry-cond statement}. Furthermore, by construction $\Vert v_1 \Vert_{L^\infty(\Omega \cap B_1)}\leq \eta$ and $|f_1|\leq \eta \phi^{\theta_+ +\e -2}$ in $\Omega \cap B_1$. \\

Thus, the function $u =v_1 -\Big(d^{\theta_+} +d^{\theta_+ +\e}\Big)$ satisfies the hypotheses of Lemma \ref{lem-comparison}, implying that $v_1 \leq \phi^{\theta_+} -\phi^{\theta_+ +\e}$ in $\Omega \cap B_{\rho}$. By applying the previous argument to  $-v$ and combining it with \eqref{eq equivalence phi}  we derive \eqref{eq boundary estimate}.\\

We complete the proof deducing \eqref{eq holder estimate}. Let us consider first the case when $\kappa <0$. Given $x_0 \in \Omega$ with $\text{dist}(x_0,\partial \Omega)= 2r<\rho$, we can combine  \eqref{eq boundary estimate} with interior Schauder estimates to deduce	
\begin{eqnarray*}
	[ v ]_{C^{\theta_+}(B_\frac{r}{2}(x_0))}&\leq& C\Big( \frac{1}{r^{\theta_+}} \Vert v\Vert_{L^\infty(B_r(x_0))}+ r^{2-\theta_+} \left\| \frac{\kappa v}{d^2}-f\right\|_{L^\infty(B_r(x_0))}\Big)\\
	&\leq& C\Big(\Vert v\Vert_{L^\infty(\Omega \cap B_1)}+ \Vert \phi^{2-\theta_+ -\e} f \Vert_{L^\infty(\Omega \cap B_1)}\Big).
\end{eqnarray*}

The estimate \eqref{eq holder estimate} follows from the previous inequality combined with a covering argument. See, e.g., Corollary 3.5 and Lemma B.2 in \cite{Kuk22}. In the case $\kappa \geq 0$, we can argue analogously to derive \eqref{eq lip estimate} see, e.g., \cite[Corollary 3.5]{Kuk22}.
\end{proof}

\begin{lemma}\label{lemma uni approx poly} 
	Let $M\in \N$, $\kappa \geq -\frac14$, $\beta \in (1,2)$, and let $\Omega \subset \R^n$ be a $C^{\beta}$ domain. Let $d$ be the regularized distance to $\partial\Omega$, given by Lemma \ref{lemma generalized distance}. Let $v \in C(\Omega\cap B_1)$ be a solution of
	\begin{equation}\label{eq compactness}
			L_\kappa (v)(x) =  \big(R_1(x)+g_1(x)\big)\left(\frac{d(r x)}{r}\right)^{\theta_+-1}+\big(R_2(x)+g_2(x)\big)\left(\frac{d(r x)}{r}\right)^{\theta_--1}\quad \textrm{in}\quad  {\textstyle\frac{1}{r}}\Omega\cap B_1,
	\end{equation}
with  $R_i \in \mathbf{P}_{M}$, and $g_i\in L^\infty({\textstyle\frac{1}{r}}\Omega\cap B_1)$ for $i=1,2$ and with $R_2=g_2=0$ if $\theta_- \leq0$. Then, there exist $\rho_0>0$ and $C>0$ only depending on $M$ and $\beta$ such that 
 then  for any $r \in (0,\rho_0)$,	
	\begin{equation}\label{eq polynomial bound}
		\Vert R_i\Vert_{L^\infty\big({\textstyle\frac{1}{r}}\Omega\cap B_1\big)}\leq C\big(	\Vert g_1\Vert_{L^\infty\big({\textstyle\frac{1}{r}}\Omega\cap B_1\big)}+\Vert g_2\Vert_{L^\infty\big({\textstyle\frac{1}{r}}\Omega\cap B_1\big)}+\Vert v\Vert_{L^\infty\big({\textstyle\frac{1}{r}}\Omega\cap B_1\big)}\Big),
	\end{equation}
for $i=1,2$.
\end{lemma}

\begin{proof}
	Arguing by contradiction, we assume sequences of domains $\Omega_l$ with $\Vert \partial \Omega_l \cap B_2\Vert_{C^\beta}\leq 1$, a sequence  of radii $\{r_l\}_{l\in \N}$ with $r_l\to 0$, and of polynomials $\{R_i^l\}_{l\in \N}$, functions $g_i^l \in L^\infty\big( \frac{1}{r_l}\Omega \cap B_1\big)$ for $i=1,2$ and solutions $u_l$ to \eqref{eq compactness} such that \eqref{eq polynomial bound} does not hold. Thus, by dividing both sides of \eqref{eq compactness} and \eqref{eq polynomial bound} by $\Vert R_1^l\Vert_{L^\infty\big( \frac{1}{r_l}\Omega \cap B_1\big)}+	\Vert R_2^l\Vert_{L^\infty\big( \frac{1}{r_l}\Omega \cap B_1\big)}$ we obtain a sequence of functions $\{\tilde{v}_l\}$ satisfying
	\begin{equation*}
	 	L_\kappa (\tilde{v}_l)(x) =  \Big(\tilde{R}_1^l(x)+\tilde{g}_1^l(x)\Big)\Big(\frac{d_l(r x)}{r}\Big)^{\theta_+-1}+\Big(\tilde{R}_2^l(x)+\tilde{g}_2^l(x)\Big)\Big(\frac{d_l(r x)}{r}\Big)^{\theta_--1},\quad \text{in}\,\, \frac{1}{r_l}\Omega \cap B_1,
	\end{equation*}
with  $\Vert \tilde{g}_i^l\Vert_{L^\infty(\Omega \cap B_\rho)} \to 0$ as $l\to \infty$ for $i=1,2$,
\begin{equation}\label{eq vanishing ul}
	\limsup_{l\to \infty} \Vert u_l\Vert_{L^\infty(B_1\cap \frac{1}{r_l}\Omega)}\to 0,
\end{equation}
and
	$$\Vert \tilde{R}_1^l\Vert_{L^\infty\big( \frac{1}{r_l}\Omega \cap B_1\big)}+	\Vert \tilde{R}_2^l\Vert_{L^\infty\big( \frac{1}{r_l}\Omega \cap B_1\big)} =1.$$
	
	This last bound on the polynomials implies, bearing in mind that $\mathbf{P}_{M}$ is a finite dimensional space, that, up to substracting a further subsequence, $\tilde{R}_i^l \to R^*_i$ for $i=1,2$ with at least one $Q_i^*\neq 0$. On the other hand, in virtue of Lemma \ref{lemma generalized distance}, we have that, up to subsequences,
	\begin{equation*}
		\frac{d_l(r_l x)}{r_l}\to x_n,
	\end{equation*}
locally uniformly in $\R^n$ -by extending $d$ by zero outside $\Omega \cap B_1$. Therefore, by the stability of viscosity solutions under locally uniform convergence, we deduce that for $\lim_{l\to \infty} u_l=u_*=0$
\begin{equation*}
0=	\Delta u_* - 	\frac{\tilde{C_\gamma}}{x_n^2}u = Q_1^* x_n^{\theta_+-1}+Q_2^* x_n^{\theta_--1},
\end{equation*}
which yields a contradiction since at least one of the $Q_i^*$ is not zero and since $\theta_+ -\theta_- \notin \N$. 	
\end{proof}

We prove now a Liouville-type theorem. 

\begin{lemma}\label{lemma general Liouville}
	Let $M \in \N$, $\kappa \geq -\frac{1}{4}$, $\theta_\pm$ as in \eqref{eq powers}, and $\beta >1$ with $\beta \notin \N$. Let $Q_i \in \mathbf{P}_{M}$ for $i=1,2$ with $Q_2=0$ if $\theta_- \leq0$. If $u$ is such that	
	\begin{equation}\label{eq Liouville 2}
		\begin{cases}
			L_\kappa u = Q_1(x)x_n^{\theta_+-1}+Q_2(x)x_n^{\theta_--1} \quad \textrm{in}  \quad \{x_n>0\}\\
			\Vert u\Vert_{L^\infty(B_R)} \leq CR^{{\theta_+}+\beta-1}.
		\end{cases}
	\end{equation}
	
 Let also us assume that $u$ satisfies one of the following conditions 
	\begin{enumerate}
		\item either	
		\begin{equation}\label{eq lower gamma liouville}
			\begin{array}{rl}
				\displaystyle	\lim_{x_n \to 0^+}	\frac{u(x',x_n)}{x_n^{\theta_-}} = 0,& \quad \textrm{if}\quad \kappa>-{\textstyle \frac14} \vspace{1mm} \\
				\displaystyle 	\lim_{x_n \to 0^+}	\frac{u(x',x_n)}{ \ln(x_n) x_n^\frac{1}{2}} = 0, & \quad \textrm{if}\quad \kappa=-{\textstyle \frac14} .
			\end{array}
		\end{equation}
		
		\item or 
		\begin{equation}\label{eq higher gamma liouville} 
			\lim_{x_n \to 0^+}		\frac{u(x',x_n)}{x_n^{\theta_+}} = 0.
		\end{equation}
	\end{enumerate}	
	
	Then, $ u = P_1(x)x_n^{\theta_+}+ P_2(x)x_n^{\theta_-+1}$,  where $P_1 \in \mathbf{P}_{\lfloor \beta \rfloor -1 }$ and  $P_2 \in \mathbf{P}_{\lfloor \beta +\theta_+-\theta- \rfloor -2 }$. Furthermore, if \eqref{eq lower gamma liouville}  holds, then    $P_2=0$ whenever $Q_2=0$. On the other hand, if \eqref{eq higher gamma liouville}  holds, $P_1=0$ whenever $Q_1=0$.
\end{lemma}

\begin{proof}
We will assume throughout the proof $\kappa <0$, since the case $\kappa \geq 0$ is completely analogous. \\
	
	Given any $R>1$, let us define $u_R(x) = R^{-\theta_+-\beta+1}u(R x)$. Thanks to the growth bound in \eqref{eq Liouville 2}, we have that $\Vert u_R\Vert_{L^\infty(B_1)} \leq C$. Noticing that 
	$$L_k u_R = R^{2-\beta} Q_1(R x) x_n^{\theta_+-1}+ R^{2-\beta-\theta_++\theta_-} Q_2(R x) x_n^{\theta_--1},$$
	we can use Lemma \ref{lemma uni approx poly}  in the flat domain $\Omega = \R^n_+$, to deduce
	\begin{equation}\label{eq bound polyscale}
R^{2-\beta}\Vert Q_1(R x)\Vert_{L^{\infty}(B_1)} + R^{2-\beta-\theta_++\theta_-}\Vert Q_2(R x)\Vert_{L^{\infty}(B_1)}\leq C,
	\end{equation}
	for any $R>1$, implying that 
	\begin{equation}\label{eq degree poly}
		Q_1 \in \mathbf{P}_{\lfloor \beta \rfloor -2}, \quad Q_2 \in \mathbf{P}_{\lfloor \beta + \theta_+ -\theta_- \rfloor -2}.
	\end{equation}
	
	Now, taking $\e \in (0, 1-\theta_+-\theta_-)$ in \eqref{eq holder estimate}, we deduce from \eqref{eq bound polyscale} and Lemma \ref{lemma basic-boundary regularity} that $	[ u_R ]_{C^{\theta_+}(B_\frac{1}{2})} \leq C(1+\Vert Q_1\Vert_{L^\infty(B_1)}+\Vert Q_2\Vert_{L^\infty(B_1)})$. This implies that
	\begin{equation}\label{eq first quotient}
		[ u ]_{C^{\theta_+}\big(B_{R/2}\big)}\leq C R^{\beta-1}.
	\end{equation}
	
	Arguing as in the proof of \cite[Theorem 3.10]{AR}, we consider the quotients
	\begin{equation}\label{eq quotient}
		u_1(x) = \frac{u(x+\tau h)-u(x)}{h^{\theta_+}}
	\end{equation}
	with $h\in (0,R/2)$, $x\in \R^n_+$, and $\tau \in \mathbb{S}^{n-1}$ with $\tau_n=0$. From \eqref{eq first quotient}, we have that $\Vert u_1\Vert_{L^\infty(B_R)} \leq CR^{\beta-1}$. On the other hand, since $\tau$ is orthogonal to $e_n$, we have that $u_1$ satisfies the boundary condition \eqref{eq lower gamma liouville} together with
	\begin{equation}\label{eq Liouville 2.1}
	\begin{cases}
		L_\kappa u_1 = \tilde{Q}_1(x)x_n^{\theta_+-1}+\tilde{Q}_2(x)x_n^{\theta_--1}, \quad \{x_n>0\}\\
		\Vert u_1\Vert_{L^\infty(B_R)} \leq CR^{\beta-1},\\
	\end{cases}
\end{equation}
where $\tilde{Q}_i = \frac{Q_i(x+\tau h)-Q_i(x)}{h^{\theta_+}}$. Arguing as before, we conclude that $[ u_1 ]_{C^{\theta_+}\big(B_{R/2}\big)}\leq C R^{\beta-1-\theta_+}$. By repeating this argument inductively $k$ times, until $k \theta_+ > \beta-1$, we deduce that
\begin{equation*}
	[ u_k ]_{C^{\theta_+}\big(B_{R/2}\big)}\leq C R^{\beta-1- k \theta_+} \to 0,
\end{equation*}
as $R\to \infty$, which, by the boundary condition, implies that $u_k =0$. Therefore, we conclude iterating backwards the difference quotients (see the proof of \cite[Theorem 3.10]{AR}) that
	\begin{equation}\label{eq polynomialform u}
		u(x',x_n) = \sum_{|\lambda|\leq M} C_\lambda(x_n) (x')^\lambda,
	\end{equation}
	with $M = \lfloor \beta-1\rfloor$ for some functions $C_{\lambda}(x_n)$. Our next step consists in proving by induction that
	$$C_\lambda(x_n)= P_{1,\lambda}(x_n)(x_n)_+^{\theta_+}+P_{2,\lambda}(x_n)(x_n)_+^{\theta_-+1},$$
	for polynomials $P_{1,\lambda}$ and $P_{2,\lambda}$ satisfying, furthermore, that $P_{2,\lambda}(x_n)=0$ if $Q_2=0$ whenever \eqref{eq lower gamma liouville} holds, and $P_{1,\lambda}(x_n)=0$ if $Q_1=0$ when we have \eqref{eq higher gamma liouville} instead.\\
	
	For the base case,  if we take tangential derivatives of maximal order $\lambda_0$, i.e., $|\lambda_0| = M$, we conclude, bearing in mind \eqref{eq degree poly}, that $C_{\lambda_0}(x_n)$ satisfies	
	\begin{equation*}
		\begin{cases}
		L_{\kappa,1}(C_{\lambda_0})=0\\
			C_{\lambda_0}(x_n)=0, 
		\end{cases}
	\end{equation*}
	where $L_{\kappa,1}(y) = -y''+\frac{\kappa}{x_n^2}y$, i.e., the $1$ dimensional version of $L_\kappa$.\\
	
	Combining Lemma \ref{lemma one dimensional problem}, and noticing from the decomposition \eqref{eq polynomialform u} that $C_{\lambda}$ inherits the corresponding boundary condition from $u$, we deduce that
	$$C_{\lambda_0}(x_n)=C_{1,\lambda_0}x^{\theta_+},$$
	with $C_{1,\lambda_0}=0$ if \eqref{eq lower gamma liouville} holds instead. This proves the base case.\\
	
	Let us assume now the result for $|\lambda|> j$ and show that it holds for $\beta$ with $|\beta|= j$. Taking derivatives on both sides of \eqref{eq polynomialform u} with respect to the tangential variables $x'$, and setting up the notation  $v=\partial^\beta u$, we deduce from the inductive hypothesis that	
	\begin{equation}\label{eq structure v}
		u(x',x_n)= \beta! C_{\beta}(x_n)+\sum_{\{ \lambda: \lambda>\beta\}} (x')^{\lambda-\beta}\Big( P_{1,\lambda}(x_n)(x_n)_+^{\theta_+}+P_{2,\lambda}(x_n)_+^{\theta_-+1}\Big),
	\end{equation}
	with $P_{i,\lambda}=0$ if $Q_i=0$ for $i=1,2$, provided that the right corresponding boundary behavior is enforced. Thus, differentiating \eqref{eq Liouville 2} and applying $L_{\kappa,1}$ to the representation \eqref{eq structure v} yields	
	\begin{eqnarray}\nonumber
		\beta! L_{k,1}(C_{\beta}(x_n))&=&\partial^\beta Q_1(x)x_n^{{\theta_-}-1}+ \partial^\beta Q_2(x)x_n^{{\theta_+}-1}\\\nonumber
		&&- \sum_{\{ \lambda: \lambda>\beta\}} (\Delta (x')^{\lambda-\beta})\Big( P_{1,\lambda}(x_n)x_n^{{\theta_+}}+P_{2,\lambda}(x_n)x_n^{\theta_-+1}\Big)\\\label{eq equation derivative}
		&&- \sum_{\{ \lambda: \lambda>\beta\}} ( (x')^{\lambda-\beta})\Big( P_{1,\lambda}''(x_n)x_n^{\theta_+}+P_{2,\lambda}''(x_n)x_n^{\theta_-}\Big).
	\end{eqnarray}
	
	Therefore, since  $ L_{\kappa,1}(C_{\beta}(x_n))$ does not depend on $x'$, we deduce from \eqref{eq equation derivative} that
	\begin{equation}
	L_{\kappa,1}(C_{\beta}(x_n))=\tilde{P}_{1}(x_n)x_n^{{\theta_+}-1}+ \tilde{P}_{2}(x_n)(x_n)x_n^{{\theta_-}-1},
	\end{equation}
	for some polynomials ${\tilde{P}_{1}}$ and $\tilde{P}_{2}$, with $\tilde{P}_{i}=0$ if $Q_i =0$ for $i=1,2$. Then, combining again Lemma \ref{lemma one dimensional problem}, with either  \eqref{eq lower gamma liouville} or  \eqref{eq higher gamma liouville}, the decomposition \eqref{eq polynomialform u}, and the linearity of $L_{\kappa,1}$ we deduce
	$$ C_{\beta}(x_n) = P_{1,\beta}(x_n)x_n^{\theta_+}+ P_{2,\beta}(x_n)x_n^{\theta_-+1},$$
	with $ P_{i,\beta}=0$ if $Q_i =0$ for $i=1,2$.\\
	
	This completes the inductive argument implying
	\begin{equation}\label{eq structure u}
		u(x) = P_1(x)x_n^{\theta_+}+ P_2(x)x_n^{\theta_-+1}
	\end{equation} 
	for some polynomials $P_1$ and $P_2$, with $ P_{i}=0$ if $Q_i =0$.\\ 
	
	Lastly, since $\theta_+-\theta_- \notin \N$, each term in \eqref{eq structure u} must satisfy the same growth bound as $u$. This implies that $P_1 \in \mathbf{P}_{\lfloor \beta \rfloor-1}$ and that $P_2 \in \mathbf{P}_{\lfloor \beta +\theta_+-\theta- \rfloor -2 }$, finishing the proof.
\end{proof}

\section{Higher order boundary Schauder estimates for solutions to the linearized equation}
\label{sec4}

This section is devoted to the proof of a generalized version of Theorem  \ref{thm schauder}, which is carried out using a blow-up in the same spirit of \cite{AR}. We divide the argument into two parts, addressing first the case where $\partial \Omega \cap B_1$ is a $C^{\beta}$ surface with $\beta \in (1,2)$ and then proving the higher order version of this result. 
 
 \begin{theorem}\label{theorem initial reg linear}
 	Let $\kappa \geq -\frac{1}{4}$, $\theta_\pm$ as in \eqref{eq powers}, and $\beta \in(1,2)$ with $\beta+\theta_+-\theta_- \notin \N$. Let $\Omega \subset \R^n$ such that $\partial \Omega \cap B_2$ a $C^{\beta}$ surface,  let $f$ and $g$ satisfying  $\Vert f\Vert_{C^{\beta-1}(\overline{\Omega} \cap B_1)}, \Vert g\Vert_{C^{\beta-1+\theta_+-\theta_-}(\overline{\Omega} \cap B_1)}\leq 1$ and $f, g|_{\partial\Omega}\equiv0$, where $g=0$ if $\theta_- \leq0$. Let $v \in  C(\Omega \cap B_1)$ be a solution to
 	\begin{equation}\label{eq linearized base case}
 			L_\kappa(v)= fd^{\theta_+-2}+gd^{\theta_--2},\quad \text{in}\,\, \Omega \cap B_1,
 	\end{equation}
 	satisfying 
 	\begin{equation}\label{bdry-cond schauder}
 		\begin{array}{rl}
 			\displaystyle \lim_{x\to \partial\Omega} \frac{v}{d^{\theta_-}} = 0 & \quad \textrm{if}\quad \kappa>-{\textstyle \frac14} \vspace{1mm} \\
 			\displaystyle \lim_{x\to \partial\Omega} \frac{v}{d^{1/2} |\log d|} = 0 & \quad \textrm{if}\quad \kappa=-{\textstyle \frac14}.
 		\end{array}
 	\end{equation}

 	Then, there exists a constant $C>0$ and a radius $\rho>0$, such that for every $x_0 \in \partial \Omega \cap B_\rho$, there are constants $\lambda_{1,x_0}, \lambda_{2,x_0} \in \R$ such that
 	\begin{eqnarray}\label{eq auxiliary linear lower order} 
 		|v(x)-\lambda_{1,x_0}d^{\theta_+}(x)-\lambda_{2,x_0}(x)d^{\theta_-+1}(x)|\leq C |x-x_0|^{{{\theta_+} + \beta-1}}\quad  \mbox{for $x\in B_{\rho}(x_0)\cap \Omega$,}
 	\end{eqnarray}	
 	with $\lambda_{2,x_0}=0$ if $g=0$. Here the constants $C$ and $\rho$ depend only on $\kappa$, $\beta$, and $n$.
 \end{theorem}

	\begin{proof}
		\noindent {\it Preparations:} Without loss of generality, let us take $x_0=0$ and let us extend $v$ and $d$ by zero outside $\Omega\cap B_1$. So, if we contradict  the inequality \eqref{eq auxiliary linear lower order}, we deduce the existence of a sequence of domains $\Omega_j$ with $\Vert \Omega_j\Vert_{C^{\beta}}\leq 1$ and, in consequence, of generalized distance functions $d_j$ such that for any sequence of constants $\lambda_{i,j} \in \R$ (with $k_{2,j}=0$ if $\theta_- \leq0$), we have that
		\begin{equation}\label{eq contradiction statement lower linear}
			\sup_{r >0} \sup_{j\in \N} \frac{1}{r^{{{\theta_+} + \beta-1}}} \Vert v_j-\lambda_{1,j} d_j^{\theta_+}- \lambda_{2,j}d_j^{\theta_-+1}\Vert_{L^\infty(B_r)}=\infty,
		\end{equation}
		where $v_j$ satisfies \eqref{bdry-cond schauder} and
		\begin{equation}\label{eq seq linearized lower BVP}
		L_\kappa(v_j)= f_jd^{\theta_+-2}+g_jd^{\theta_--2},\quad \text{in}\,\, \Omega \cap B_1,\\
		\end{equation}
		with $\Vert f_{j}\Vert_{C^{\beta-1}(B_1\cap \overline{\Omega_j})}, \Vert g_{j}\Vert_{C^{\beta-1+\theta_+-\theta_-}(B_1\cap \overline{\Omega_j})}\leq 1$ with $f_j, g_j|_{\partial\Omega_j}\equiv0$  for $j\in \N$ and with $g_{j}=0$ if $\theta_- \leq0$.\\
		
		\medskip
		
		\textit{Step one:}  We start selecting  $\lambda_{1,j,r}, \lambda_{2,j,r}$ as the the least square coefficients of the projection of $v_j$ onto the subspace spanned by $\{d_j^{\theta_+},d_j^{\theta_-+1}\}$ in $L^2(B_r)$. In this case, they are characterized by the orthogonality conditions
		\begin{equation}\label{eq orthogonality der blow-up lower}
			\int_{B_r} \Big(v_j -\lambda_{1,j,r}d_j^{\theta_+}- \lambda_{2,j,r}d_j^{\theta_-+1}\Big)(c_1d_j^{\theta_+}+c_2d_j^{\theta_-+1})=0,
		\end{equation}
		for any $(c_1, c_2)\in  \R^2$.\\

		Let us define
		$$\theta(r) =  \sup_{\rho >r }\sup_{j\in \N} \frac{1}{r^{{{\theta_+} + \beta-1}}} \big\| v_j -\mu_{1,j} d_j^{\theta_+}- \mu_{2,j} d_j^{\theta_-+1}\big\|_{L^\infty(B_r)},$$
		with $\mu_{1,j}= \lambda_{1,j,r}$ and  $\mu_{2,j}= \lambda_{2,j,r}$.\\
		
		By Lemma \ref{eq unboundedness of theta} we have that $\lim_{r\to 0^+} \theta(r)=\infty$. Thus, we can pick a sequence $\{r_j\}_{j\in \N}$  with $r_j\to 0^+$ as $j\to \infty$ such that 	
		\begin{equation}\label{eq blowup sequence der lower}
			w_{j}(x):= \frac{ v_{j}(r_j x)-\mu_{1,j}d_j^{\theta_+}(r_j x)-\mu_{2,j}d_j^{\theta_-+1}(r_j x)}{r_j^{{{\theta_+}  +\beta-1}}\theta(r_j)},
		\end{equation}
		satisfies $ \Vert w_j\Vert_{L^\infty(B_1)} \in [\frac{1}{2},1]$ for $l\in \N$.\\		
		
		Additionally, Lemma \ref{lemma growth property} provides the growth bound
		\begin{equation}\label{eq growthcond der low}
			\Vert w_j \Vert_{L^\infty(B_R)}\leq C R^{{\theta_+}+\beta-1}.
		\end{equation}	
		
		\medskip
		
		\noindent{\it Step two:} We proceed to show that the sequence $\{w_j\}_{j\in \N}$ satisfies, up to a small perturbation, a PDE of the form \eqref{eq boundary linearized pde statement}. Applying Lemma \ref{lemma generalized distance}, we can easily compute the Laplacian of the auxiliary term $\mu_{1,j}d_j^{\theta_+}(r_jx)+\mu_{2,j}d_j^{\theta_-+1}(r_jx)$ as follows
		\begin{eqnarray}\nonumber
		\frac{1}{r_j^{{\theta_+}+\beta-1}}\Delta  \Big(\mu_{1,j}d_j^{\theta_+}(r_jx)+\mu_{2,j}d_j^{\theta_-+1}(r_j x)\Big)		&=&\frac{1}{r_j^{\theta_++\beta-1}}\frac{r_j^2\kappa}{d_j^2(rx)}\Big(\mu_{1,j} d_j^{\theta_+}(r_jx)+\mu_{2,j}d_j^{\theta_-+1}(r_jx)\Big)\\ \nonumber
		&&+ l_j  \Big( \frac{d_j(rx)}{r}\Big)^{\theta_--1}+  h_{1,j}(x) \Big( \frac{d_j(rx)}{r}\Big)^{\theta_++\beta-3}\\ \label{eq first order error linear lower}
		&&+ h_{2,j}(x) \Big( \frac{d_j(rx)}{r}\Big)^{\theta_-+\beta-2} 
	\end{eqnarray}
		where $l_j \in \R$, $\Vert h_{1,j}\Vert_{L^\infty(B_2\cap \overline{\Omega_j})}\leq \frac{C}{\theta(r_j)}$ and $\Vert h_{2,j}\Vert_{L^\infty(B_2\cap \overline{\Omega_j})}\leq \frac{C r_j^{\theta_--\theta_++1}}{\theta(r_j)}$. We remark that the latter quantity goes to zero as $j\to \infty$ since $\theta_+-\theta_- \leq 1$.\\
		
		By combining \eqref{eq first order error linear lower} with \eqref{eq seq linearized lower BVP}, we deduce
		\begin{eqnarray}\notag
			L_\kappa ( w_j) &=& l_j  \Big( \frac{d_j(rx)}{r}\Big)^{\theta_--1}+\\ \label{eq almost solution linear lower}
			&&+  h_{1,j}(x) \Big( \frac{d_j(rx)}{r}\Big)^{\theta_++\beta-3}+ h_{2,j}(x) \Big( \frac{d_j(rx)}{r}\Big)^{\theta_-+\beta-2}  \\
			 && + \tilde{f}_j(x) \Big(\frac{d_j(r_jx)}{r_j}\Big)^{\theta_+-2}+\tilde{g}_j(x)\Big(\frac{d_j(rx)}{r}\Big)^{\theta_--2}		
		\end{eqnarray}
		where $ \tilde{f}_{j}= \frac{f_j(r_j x)}{\theta(r_j)r_j^{\beta-1}}$ and  $ \tilde{g}_{j}= \frac{g_j(r_j x)}{\theta(r_j)r_j^{\beta+\theta_+-\theta_--1}}$ and with both vanishing uniformly as $j\to \infty$ thanks to the hypotheses on $f_j$ and $g_j$. In particular, we are fundamentally exploiting that $|f_j|\leq C d_j^{\beta-1}$ and that $|g_j|\leq Cd_j^{\beta-1+\theta_+-\theta_-}$.\\
		
		On the other hand, let us notice that the hypotheses \eqref{bdry-cond schauder} force $w_j$ to satisfy the boundary conditions \eqref{bdry-cond statement}. With the aim of applying Lemma \ref{lemma basic-boundary regularity}, we take $\varepsilon \in (0,\min\{\beta-1, 1+\theta_--\theta_+\})$ in \eqref{eq boundary estimate} so that the right hand side in  \eqref{eq almost solution linear lower} remains uniformly bounded after multiplying by $d_j^{2-\theta_+-\e}$. Thus,  in virtue of Lemma \ref{lemma basic-boundary regularity} and Lemma \ref{lemma uni approx poly}, the solutions $\{w_j\}_{j\in \N}$ to the sequence of problems \eqref{eq almost solution linear lower} are equicontinuous and, therefore, $w_j \to w_0$ locally uniformly as $j\to \infty$ -up to subsequences. So, thanks to  Lemma \ref{lemma uni approx poly},  and \eqref{eq growthcond der low} we deduce that		
		\begin{equation}\label{eq limiting equation linear low}
			\begin{cases}
				L_\kappa w_0 = l x_n^{\theta_--1}, \quad \mbox{ in $\{x_n>0\}$},\\
				\Vert w_0\Vert_{L^\infty(B_R\cap \R^n_+)} \leq CR^{{\theta_+}+\beta-1}.
			\end{cases}
		\end{equation}

		Additionally, thanks to the uniform convergence, the inequality \eqref{eq boundary estimate} is satisfied by each $w_j$ implies that
		\begin{equation}\label{eq boundary linear lower}
			|w_0|\leq C x_n^{\theta_+}.
		\end{equation}

		\medskip	
		\noindent{\it Step 3:} Since $w_0$ satisfies \eqref{eq limiting equation linear low} and \eqref{eq boundary linear lower}, Lemma  \ref{lemma general Liouville} applies, implying that
		$$w_0= \lambda_1 (x_n)_+^{\theta_+}+\lambda_2 (x_n)_+^{\theta_-+1}$$ 
		with $\lambda_1,\lambda_2 \in \R$.  However, taking limits in the orthogonality conditions \eqref{eq orthogonality der blow-up lower} as $j\to \infty$, yields 
		
		\begin{equation*}
			\int_{B_1} w_0  \Big(c_1(x_n)_+^{\theta_+}+c_2(x_n)_+^{\theta_-+1}\Big)=0,
		\end{equation*}
			for any $(c_1, c_2)\in  \R^2$.	 Thus, $w_0=0$ contradicting that $\Vert w_0\Vert_{L^\infty(B_1)}\geq \frac{1}{2}$.
	\end{proof}

In the context of  in Theorem \ref{thm schauder}, if $\partial \Omega\cap B_1 \in C^{\beta}$ with $\beta>2$, since $f|_{\partial\Omega}\equiv0$, we have that $fd^{\mu-2}=\tilde{f}d^{\mu-1}$ where $\tilde f = f/d$ is one degree less regular than $f$ up to the boundary. Based on this observation, we prove the following higher order version of Theorem \ref{theorem initial reg linear}.

\begin{theorem}\label{thm higher reg linear}
	Let $\kappa \geq -\frac{1}{4}$, $\theta_\pm$ as in \eqref{eq powers},  let $\beta >2$ with $\beta \notin \N$. Let $\Omega \subset \R^n$ with $\partial \Omega \cap B_1$ a $C^{\beta}$ surface,  let $f$ and $g$ satisfying $\Vert f\Vert_{C^{\beta-2}(\overline{\Omega} \cap B_1)}, \Vert g\Vert_{C^{\beta-2+\theta_+-\theta_-}(\overline{\Omega} \cap B_1)}\leq 1$, with $g=0$ if $\theta_- \leq0$. Let $v \in  C(\overline{ \Omega}\cap B_1)$ be a solution to
	\begin{equation}\label{eq higher linearized base case}
		L_\kappa(v)= fd^{\theta_+-1}+g d^{\theta_--1},\quad \text{in}\,\, \Omega \cap B_1,
	\end{equation}
	satisfying 
	\begin{equation}\label{higher bdry-cond schauder}
		\begin{array}{rl}
			\displaystyle \lim_{x\to \partial\Omega} \frac{v}{d^{\theta_-}} = 0 & \quad \textrm{if}\quad \kappa>-{\textstyle \frac14} \vspace{1mm} \\
			\displaystyle \lim_{x\to \partial\Omega} \frac{v}{d^{1/2} |\log d|} = 0 & \quad \textrm{if}\quad \kappa=-{\textstyle \frac14}.
		\end{array}
	\end{equation}
	
	Then, there exists a constant $C>0$ and a radius $\rho>0$, such that for every $x_0 \in \partial \Omega \cap B_\rho$, there exists polynomials $Q_{1,x_0} \in \mathbf{P}_{\lfloor \beta \rfloor-1}$ and $Q_{2,x_0} \in \mathbf{P}_{\lfloor \beta +\theta_+-\theta_- \rfloor -2 }$ such that
	\begin{eqnarray}\label{eq linear higher order} 
		\Big|v(x)-Q_{1,x_0}(x)d^{\theta_+}(x)-Q_{2,x_0}(x)d^{\theta_-+1}(x)\Big|\leq C |x-x_0|^{{{\theta_+} + \beta-1}}\quad  \mbox{for $x\in B_{\rho}(x_0)\cap \Omega$},
	\end{eqnarray}		
with $Q_{2,x_0}=0$ if $g=0$.	
\end{theorem}

\begin{proof}	
	\noindent {\it Preparations:} Let us assume again $x_0=0$ and let us extend $v$ and $d$ by zero in $\Omega^c\cap B_1$. By contradicting the inequality \eqref{eq linear higher order}, we deduce the existence of a sequence of domains $\Omega_j$ with $\Vert \Omega_j\Vert_{C^{\beta}}\leq 1$ and, in consequence, of generalized distance functions $d_j$ such that for any sequence of  polynomials $P_{1,j} \in \mathbf{P}_{\lfloor \beta \rfloor -1}$ and $P_{2,j} \in \mathbf{P}_{\lfloor \beta +\theta_+-\theta_- \rfloor -1}$
	\begin{equation}\label{eq contradiction statement linear}
		\sup_{r >0} \sup_{j\in \N} \frac{1}{r^{{{\theta_+} + \beta-1}}} \Big\Vert v_j-P_{1,j} d_j^{\theta_+}- P_{2,j}d_j^{\theta_-+1}\Big\Vert_{L^\infty(B_r)}=\infty,
	\end{equation}
	with $v_j$ satisfying \eqref{higher bdry-cond schauder} and
	\begin{equation}\label{eq seq linearized BVP}
			L_\kappa(v_j)= f_{1,j}d_j^{\theta_+-1}+f_{2,j}d_j^{\theta_--1}\quad \text{in}\,\, \Omega \cap B_1
	\end{equation}
	and where $\Vert f_{1,j}\Vert_{C^{\beta -2}(B_2\cap \overline{\Omega_j})}\leq 1$ and $\Vert f_{2,j}\Vert_{C^{\beta +\theta_+-\theta_- -2}(B_2\cap \overline{\Omega_j})}\leq 1$ for $j \in \N$. Furthermore, we take $f_{2,j}= P_{2,j}=0$ if $g=0$.
		
	\medskip
	
	\textit{Step one:} We select the least squares polynomials ${(Q_{1,j,r}, Q_{2,j,r}) \in \mathbf{P}_{\lfloor \beta \rfloor -1} \times \mathbf{P}_{\lfloor \beta +\theta_+-\theta_- \rfloor -2}}$, which are characterized by the orthogonality conditions
	\begin{equation}\label{eq orthogonality der blow-up k}
		\int_{B_r} (v_j -Q_{1,j,r} d_j^{\theta_+}- Q_{2,j,r}d_j^{\theta_-+1})(P_1(x)d_j^{\theta_+}+P_2(x)d_j^{\theta_-+1})=0,
	\end{equation}
	for any $(P_1, P_2)\in  \mathbf{P}_{\lfloor \beta \rfloor -1} \times \mathbf{P}_{\lfloor \beta +\theta_+-\theta_- \rfloor -2}$ for $r>0$ and $j\in \N$.\\

	Let us define
	$$\theta(r) =  \sup_{\rho >r }\sup_{j\in \N} \frac{1}{r^{{{\theta_+} + \beta-1}}} \Big\Vert v_j -Q_{1,j,r} d_j^{\theta_+}- Q_{2,j,r}d_j^{\theta_-+1}\Big\Vert_{L^\infty(B_r)}.$$
	
	By Lemma \ref{eq unboundedness of theta} we have that $\lim_{r\to 0^+} \theta(r)=\infty$. Thus, we can pick a sequence $\{r_j\}_{j\in \N}$  with $r_j\to 0^+$ as $j\to \infty$ such that 	
	\begin{equation}\label{eq blowup sequence der}
		w_{j}(x):= \frac{ v_{j}(r_j x)-Q_{1,j,r_j}(r_jx)d_j^{\theta_+}(r_j x)-Q_{2,j,r_j}(r_jx)d_j^{\theta_-+1}(r_j x)}{r_j^{{{\theta_+}  +\beta}}\theta(r_j)},
	\end{equation}
	satisfies $ \Vert w_j\Vert_{L^\infty(B_1)} \in [\frac{1}{2},1]$ for $j\in \N$.\\		
	
	Additionally, Lemma \ref{lemma growth property} provides the growth bound
	\begin{equation}\label{eq growthcond der}
		\Vert w_j \Vert_{L^\infty(B_R)}\leq C R^{{\theta_+}+\beta-1}.
	\end{equation}	
	
	\medskip
	
	\noindent{\it Step two:} We proceed to show that the sequence $\{w_j\}_{j\in \N}$ satisfies, up to a small perturbation,  a PDE of the form \eqref{eq boundary linearized pde statement}. Let us start noticing that since we fall under the hypothesis of Lemma \ref{lemma growth property}, we can estimate the Laplacian of the auxiliary term $Q_{1,j,r} d_j^{\theta_+}+ Q_{2,j,r}d_j^{\theta_-+1}$ as follows
	\begin{eqnarray}\nonumber
	\frac{1}{\theta(r)r^{{\theta_+}+\beta-1}}\Delta  \Big(Q_{1,r}(rx)d_j^{\theta_+}(r x)+Q_{2,r}(rx)d_j^{\theta_-+1}(r x)\Big)	&=&\frac{1}{\theta(r)r^{{\theta_+}+\beta-1}}\frac{\kappa r^2}{d_j^2(rx)}Q_{1,r}(rx)d_j^{\theta_+}(r x)\\ \nonumber
	&&+\frac{1}{\theta(r)r^{{\theta_+}+\beta-1}}\frac{\kappa r^2}{d_j^2(rx)}Q_{2,r}(rx)d_j^{\theta_-+1}(r x)\\\nonumber
		&&+P_{1,r}(x)\Big( \frac{d_j(rx)}{r}\Big)^{\theta_+-1}\\\nonumber
		&& +P_{2,r}(rx) \Big( \frac{d_j(rx)}{r}\Big)^{\theta_--1}\\ \label{eq first order error linear}
		&&+m_*(r,x)
	\end{eqnarray}
where $P_{i,r} \in \mathbf{P}_{M}$ for $i=1,2$ for some $M \in \N$ only depending on $\beta$ and $\kappa$ and where $|m_*(x,r)|\leq Cm(r)$ with $m$ depending only on $\beta$ and $\kappa$ and such that $m(r)\to 0$ as $r\to 0^+$.\\
	
	By combining \eqref{eq seq linearized BVP}, \eqref{eq first order error linear},  and exploiting the uniform regularity of the functions $f_{i,j}$ we deduce	
	\begin{eqnarray}\notag
		L_\kappa ( w_j) &=& \Big( \tilde{Q}_{1,r}(x)+g_{1,j}(x) \Big)\Big( \frac{d_j(rx)}{r}\Big)^{\theta_+-1} +\Big( \tilde{Q}_{2,r}(rx) +g_{2,j}(x)\Big) \Big( \frac{d_j(rx)}{r}\Big)^{\theta_--1}\\\label{eq almost solution linear}
		&&+g_{3,j}(x),
	\end{eqnarray}
	where $\tilde{Q}_{i,j} \in  \mathbf{P}_{M }$ for some $M\in \N$ and $\Vert g_{i,j} \Vert_{L^\infty(B_2\cap \overline{\Omega_j})}\to 0$ as $j\to \infty$ for $i=1,2,3$.\\
	
		On the other hand, let us notice that the hypotheses \eqref{higher bdry-cond schauder} force $w_j$ to satisfy the boundary conditions \eqref{bdry-cond statement}. With the aim of applying Lemma \ref{lemma basic-boundary regularity}, we take $\varepsilon \in (0,1+\theta_--\theta_+)$ in \eqref{eq boundary estimate} so that the right hand side in  \eqref{eq almost solution linear} remains uniformly bounded after multiplying by $d_j^{2-\theta_+-\e}$. Thus,  in virtue of Lemma \ref{lemma uni approx poly} and Lemma \ref{lemma basic-boundary regularity}, the solutions $\{w_j\}_{j\in \N}$ to the sequence of problems \eqref{eq almost solution linear lower} are equicontinuous and, therefore, satisfy that $w_j \to w_0$ locally uniformly. So, thanks to  Lemma \ref{lemma uni approx poly},  and \eqref{eq growthcond der low} we deduce that			
	\begin{equation}\label{eq limiting equation linear}
		\begin{cases}
			L_\kappa w_0 = Q_1(x)(x_n)_+^{\theta_+-1}+Q_2(x)(x_n)_+^{\theta_--1} \quad \mbox{ in $\{x_n>0\}$ },\\
			\Vert w_0\Vert_{L^\infty(B_R\cap \R^n_+)} \leq CR^{{\theta_+}+\beta-1}
		\end{cases}
	\end{equation}
	with $Q_i \in  \mathbf{P}_{M}$ for some $M\in \N$ and for $i=1,2$ and with $Q_2=0$ if $g=0$. Additionally, thanks to the uniform convergence, the inequality \eqref{eq boundary estimate} satisfied by each $w_j$ implies that
	\begin{equation}\label{eq boundary linear}
		|w_0|\leq C x_n^{\theta_+}.
	\end{equation}
	
	\medskip	
	\noindent{\it Step 3:} Since $w_0$ satisfies \eqref{eq limiting equation linear} and \eqref{eq boundary linear}, Lemma  \ref{lemma general Liouville} applies, implying that
	$$w_0= P_1(x)(x_n)_+^{{\theta_+}}+P_2(x)(x_n)_+^{{\theta_-+1}},$$ 
	with $P_1 \in \mathbf{P}_{\lfloor \beta\rfloor-1}$ and $P_2 \in \mathbf{P}_{\lfloor \beta +\theta_+-\theta_-\rfloor-2}$, with $P_2=0$ if $g=0$. However, taking limits in the orthogonality condition \eqref{eq orthogonality der blow-up k} as $j\to \infty$, yields	
	\begin{equation*}
		\int_{B_1} w_0  \Big(Q_1(x)x_n^{{\theta_+}}+Q_2(x)x_n^{{\theta_-+1}}\Big)=0,
	\end{equation*}
	for all  $Q_1 \in \mathbf{P}_{\lfloor \beta\rfloor-1}$ and $Q_2 \in \mathbf{P}_{\lfloor \beta +\theta_+-\theta_-\rfloor-2}$. Thus, $w_0=0$ contradicting $\Vert w_0\Vert_{L^\infty(B_1)}\geq \frac{1}{2}.$
\end{proof}

The following technical lemma shows how to translate estimates of the form \eqref{eq linear higher order} into regularity for the quotients $v/d^\mu$. 

\begin{lemma}\label{lemma quotient}
		Let $\mu>0$, let $\beta >1$ with $\beta \notin \N$, and let $\Omega \subset \R^n$ with $\partial \Omega \cap B_1$ a $C^{\beta}$ surface. Let $v \in C^{\beta}(\Omega \cap B_1)$ and let us assume that for every $z \in  \Omega \cap B_1$ there exists $Q_z \in  \mathbf{P}_{\lfloor \beta-1\rfloor}$ such that for every $y_0 \in B_\frac{1}{2}(z)\cap \Omega$ satisfying $\text{dist}(y_0)= |y_0-z|=2 r$
		\begin{equation}\label{eq uniform approx} 
			[v-Q_zd^\mu]_{C^{\beta-1}(\overline{B_r(y_0)})}\leq C r^{\mu},
		\end{equation}
		and
		\begin{equation}\label{eq linfty bound2}  
		\Vert v-Q_zd^\mu\Vert_{L^\infty({B_r(y_0)})}\leq C r^{\mu+\beta-1},
		\end{equation}

		Then, 
		\begin{equation}\label{eq reg quotient} 
			\Big\Vert \frac{v}{d^{\mu}}\Big\Vert_{C^{\beta-1}(\overline{\Omega} \cap B_{1/2})} \leq C.
		\end{equation}
	
	More in general, if there is an integer $k>\beta-1$ such that $v\in C^k(\Omega \cap B_1)$ and such that for every $\lambda \in \N^n$ with $|\lambda| =k$,
	\begin{equation}\label{eq uniform decay}
	\Vert \partial^\lambda(v-Q_zd^\mu)\Vert_{L^\infty(B_r(y_0))}\leq C_k r^{\mu+\beta-1-k},
	\end{equation}
 then
	\begin{equation}\label{eq higherder quotient 0}
		|\partial^\lambda\big(v/d^{\mu}\big)\big| \leq C_kd^{\beta-1-k}\quad \textrm{in}\quad \Omega\cap B_\frac{1}{2}.
	\end{equation}

\end{lemma}
\begin{proof}
	The proof of the first part of this lemma is essentially contained in the proof of \cite[Theorem 1.4]{AR}. Let $x_1, x_2\in B_r(y_0)$ and let $\lambda \in \N^n$ be a multi-index with $|\lambda| =  \lfloor \beta\rfloor-1$. Let $Q_z$ as in \eqref{eq uniform approx}, arguing as in the proof of \cite[Theorem 1.4]{AR}, we have the identity
	\begin{eqnarray}\notag
		\partial^\lambda (d^{-\mu}u)(x_1)-\partial^\lambda (d^{-\mu}u)(x_2)&=& \sum_{\alpha \leq \lambda} [\partial^\alpha (u-Q_zd^\mu)(x_1) - \partial^\alpha (u-Q_zd^\mu)(x_2)]\partial^{\lambda-\alpha}d^{-\mu}(x_1)\\\label{eq RA}
		&&+ \partial^\alpha (u-Q_zd^\mu)(x_2)[\partial^{\lambda-\alpha}d^{-\mu}(x_1)-\partial^{\lambda-\alpha}d^{-\mu}(x_2)],
	\end{eqnarray}
where we have used that $\partial^\lambda Q_z(x_1) = \partial^\lambda Q_z(x_2)$ since $Q_z$ has degree $\lfloor \beta-1\rfloor = \lambda$.\\

We start bounding the terms in the right hand side of \eqref{eq RA}, recalling from  Lemma \ref{lemma generalized distance} that
\begin{equation}\label{eq distance decay}
	|\partial^{\lambda-\alpha}d^{-\mu}(x_1)|\leq Cr^{|\alpha|-|\lambda|-\mu},
\end{equation} 
for any $\alpha \leq \lambda$. On top of this, we can combine \eqref{eq distance decay} with the mean value theorem to deduce
\begin{equation}\label{eq distance decay double}
	|\partial^{\lambda-\alpha}d^{-\mu}(x_1)-\partial^{\lambda-\alpha}d^{-\mu}(x_2)|\leq Cr^{|\alpha|-\mu-\beta}|x_1-x_2|^{\beta-\lfloor \beta\rfloor}.
\end{equation}
		
On the other hand, we can use interpolation \cite[Lemma 6.32]{GT} to deduce from \eqref{eq uniform approx} and \eqref{eq linfty bound2} that
\begin{equation}\label{eq holder interpol}
		[\partial^\alpha(v-Q_zd^\mu)]_{C^{\beta- \lfloor \beta\rfloor}(\overline{B_r(y_0)})}\leq C r^{\mu-|\alpha|+ \lfloor \beta\rfloor-1}
\end{equation}
for any $\alpha \leq \lambda$. In the same vein, we have that
\begin{equation}\label{eq interpol 2}
	\sum_{|\alpha| \leq \lfloor \beta\rfloor-1} r^{|\alpha|}	\Vert \partial^{\alpha}(v-Q_zd^\mu)\Vert_{L^\infty(B_r(y_0))}\leq C r^{\mu+\beta-1}.
\end{equation}
	
Hence,  by combining  \eqref{eq RA}, \eqref{eq distance decay}, \eqref{eq holder interpol}, and \eqref{eq interpol 2} we deduce
	\begin{eqnarray*}
	|\partial^\lambda (d^{-\mu}u)(x_1)-\partial^\lambda (d^{-\mu}v)(x_2)|\leq C|x_1-x_2|^{\beta-\lfloor \beta\rfloor},
\end{eqnarray*}
from where \eqref{eq reg quotient} follows.\\

Let us assume now that \eqref{eq uniform decay} holds for some $k > \lfloor \beta\rfloor-1$.  Let $x\in B_r(y_0)$ and $\lambda \in \N^n$ with $|\lambda| =  k$. Since $\partial^{\lambda}Q_z=0$, we can proceed as in \eqref{eq RA} to get
	\begin{eqnarray}\label{eq expansion der}
	\partial^\lambda (d^{-\mu}u)(x) = \partial^\lambda (d^{-\mu}v-Q_z)(x) = \sum_{\alpha \leq \lambda} \partial^\alpha (v-d^{\mu}Q)(x)\, \partial^{\lambda-\alpha} d^{-\mu}.
\end{eqnarray}

Using again \cite[Lemma 6.32]{GT}, we can interpolate \eqref{eq linfty bound2} and \eqref{eq uniform decay} to deduce
\begin{equation}\label{eq interpol 3}
	\sum_{|\alpha| \leq k} r^{|\alpha|}	\Vert \partial^{\alpha}(v-Q_zd^\mu)\Vert_{L^\infty({B_r(y_0)})}\leq C r^{\mu+\beta-1-k}.
\end{equation}

Altogether \eqref{eq distance decay}, \eqref{eq expansion der}, and \eqref{eq interpol 3} implies 
\begin{equation*}
		|\partial^\lambda (d^{-\mu}v)(x)|\leq C r^{\beta-1-k}\leq C d^{\beta-1-k}(x),
\end{equation*}
where the last inequality follows from the comparability of $d(y_0)$ and $r$.
\end{proof}

\begin{proof}[Proof of Theorem \ref{thm schauder}]
Let us start by normalizing both sides of \eqref{eq boundary linearized pde} dividing by $\|f\|_{C^{k,\alpha}(\overline\Omega\cap B_1)} + \|v\|_{L^\infty(\Omega \cap B_1)}$. Thus, in virtue of Lemma \ref{lemma quotient}, the proof will follow from verifying \eqref{eq uniform approx}. Let $\beta = k+1+\alpha$ and let us invoke Theorem \ref{theorem initial reg linear} if $k=1$ and Theorem \ref{thm higher reg linear} when $k\ge 2$, to guarantee the existence of  a polynomial $Q_{z} \in \mathbf{P}_{\lfloor \beta \rfloor-1}$ for each  $z \in \partial \Omega \cap B_1$ such that
	\begin{eqnarray}\label{eq approx higher mu} 
		\Big|v(x)-Q_{z}(x)d^{\theta_+}(x)\Big|\leq C |x-z|^{{{\theta_+} + \beta-1}}\quad  \mbox{for $x\in \Omega\cap  B_{1}(z)$}.
	\end{eqnarray}		
	
	For any $y_0 \in  \Omega \cap B_\frac{1}{2}(z)$ such that $\text{dist}(y_0)= |y_0-z|=2 r$ we define the rescaling
	\begin{equation*}\label{eq rescaling}
	v_r(x) = r^{{{-\theta_+} - \beta+1}}v(y_0+rx)-  r^{1 - \beta}Q_{z}(y_0+rx)b_r(x)^{\theta_+},
\end{equation*}
where  $b_r(x)= \frac{d(y_0+rx)}{r}$. Notice that \eqref{eq approx higher mu} implies the bound $\Vert v_r\Vert_{L^\infty(B_1)}\leq C$. On the other hand, the equation for $v_r$ can be written as 
	\begin{equation}\label{eq rescaled operator}
	 -\Delta v_r+ a_r(x) v_r=  r^{1 - \beta}\Big(f(y_0+rx)b_r(x)^{\theta_+-2}-  	r^2L_\kappa\Big(Q_{z}(y_0+r\cdot)b_r^{\theta_+}\Big)(x)\Big) 
	\end{equation}
	with $a_r(x)=\frac{\kappa}{b_r(x)^2}$. On the other hand, Lemma \ref{lemma generalized distance} implies that $\frac{1}{C}\leq |b_r|\leq C$ and $| D^{\lambda} b_r| \leq C_{\lambda}$ for $x\in B_1$ and for any multi-index $\lambda$.  So, if $\beta \in (1,2)$, classical gradient estimates yield
	\begin{eqnarray}\notag
		[ v_r ]_{C^{\beta-1}(B_{1/2})}&\leq& C\Big( \Vert v_r\Vert_{L^\infty(B_1)}+  	\Vert r^{-\beta+1}f(y_0+r \cdot)b_r^{\theta_+-2}\Vert_{L^\infty(B_1)}\Big)\\ \label{eq linf int schauder}
		&&+ 	C r^{{ 3-\beta}}\Vert L_\kappa(Q_{z}(y_0+r\cdot)b_r^{\theta_+})\Vert_{L^\infty(B_1)}.
	\end{eqnarray}
Moreover, we have from Lemma \ref{lemma generalized distance} that
\begin{equation}
	|L_\kappa (Q_z(y_0+r\cdot)b_r)|\leq C b_r(x)^{\beta+\theta_+-3}\leq C
\end{equation}
and, by the $C^{\beta-1}$ regularity on $f$, $|f(y_0+rx)|\leq Cr^{\beta-1}$. Thus, by rescaling \eqref{eq linf int schauder}, we have
\begin{eqnarray}\label{eq int schauder lower reg}
	r^{-\theta}	[ v ]_{C^{\beta-1}(B_{r/2})}\leq C.
\end{eqnarray}
	
If $\beta>2$, since $f$ vanishes identically on the boundary, we can write $f d^{\theta_+-2}= f_0d^{\theta_+-1}$ with $f_0 \in C^{\beta-2}(\overline{\Omega}\cap B_1)$. Additionally, in virtue of Lemma \ref{lemma generalized distance}, we have that
 \begin{equation}\label{eq L distance}
	L_\kappa(Q_{z}d^{\theta_+})(x)=- g(x)d(x)^{\theta_+-1}
\end{equation}
for some $g \in C^{\beta-2}(\overline{\Omega}\cap B_1)$. Altogether, these observations allow us to rewrite \eqref{eq rescaled operator} as follows
	\begin{equation*}
	-\Delta v_r+ a_r(x) v_r=  r^{2 - \beta}\Big(f_0(y_0+rx)+g(y_0+rx)\Big)b_r(x)^{\theta_+-1},
\end{equation*}
which can be rewritten in the form
	\begin{equation}\label{eq rescaled operator 2}
	\frac{-1}{b_r(x)^{\theta_+-1}}\Delta v_r+\frac{a_r(x)}{b_r(x)^{\theta_+-1}} v_r=  r^{2 - \beta}\Big(f_0(y_0+rx)+g(y_0+rx)\Big).
\end{equation}
Thus, since $b_r \in C^\beta(B_1)$ and  $\frac{1}{C}\leq |b_r|\leq C$, we can apply standard interior Schauder estimates to \eqref{eq rescaled operator 2} (see, e.g., \cite{Sim97}) to deduce
\begin{eqnarray*}
	[ v_r ]_{C^{\beta-1}(B_{1/2})}\leq C\Big( \Vert v_r\Vert_{L^\infty(B_1)}+ r^{2-\beta} 	[f(y_0+r \cdot)+g(y_0+r\cdot)]_{C^{\beta-2}(B_1)}\Big)
\end{eqnarray*}
After rescaling, we deduce
\begin{eqnarray}\label{eq int schauder rescaled}
	r^{-\theta}	[ v ]_{C^{\beta-1}(B_{r/2})}\leq C\Big(1+ 	[f+g]_{C^{\beta-2}(B_r(y_0))}\Big)\leq C,
\end{eqnarray}
which concludes the proof.
\end{proof}

We finish this section deriving a modified version of Theorem \ref{thm schauder} that will be applied to $u-c_\gamma d^{\frac{2}{2-\gamma}}$ in the next section.
	
\begin{corollary}\label{corollary schauder}
	Let $\kappa \geq -\frac14$, let $\mu =\theta_+$ if $\kappa \geq 0$ and $\mu \in \{\theta_+,\theta_-\}$ when $\kappa<0$.  Let $\Omega \subset \R^n$ be a $C^{\beta}$ domain with $\beta>2$, $f\in C^{\beta-2}(\overline\Omega\cap B_1)$, and let $d$ be the regularized distance to $\partial\Omega$ given by Lemma \ref{lemma generalized distance}.\\

	Let $v \in C(\Omega \cap B_1)$ be any solution to 
	\begin{equation}\label{eq boundary linearized pde corollary}
		-\Delta v + \kappa\,\frac{v}{d^2}=fd^{\mu-1}\quad \text{in}\quad \Omega\cap B_1.
	\end{equation}
	
	If it satisfies
	\begin{equation}\label{bdry-cond higher}
	\displaystyle \lim_{x\to \partial\Omega} \frac{v}{d^{\theta_+}} = 0,
\end{equation}
	then, $v\in C^{\beta-2}(\overline\Omega\cap B_{1/2})$ and 
	\[
	\left\|v/{d^{\mu+1}}\right\|_{C^{\beta-2}(\overline\Omega \cap B_{1/2})}\leq C\big(\|f\|_{C^{\beta-2}(\overline\Omega\cap B_1)} + \|v\|_{L^\infty(\Omega \cap B_1)} \big),
	\]
	where $C>0$ depends only on $\Omega$, $k$, $\alpha$, $n$, and $\kappa$.\\
	
	Furthermore, if there is $k>\beta-2$ such that $v\in C^k(\Omega\cap B_1)$, $f\in C^k(\Omega\cap B_1)$ and 
	\begin{equation}\label{eq decaying f}
			|\partial^{\alpha} f| \leq C_{\alpha} d^{\beta-2 -|\alpha|}, \quad \mbox{ in $\Omega\cap B_1$},
	\end{equation}
	for every $|\alpha| \in \{\lfloor \beta \rfloor-1,\cdots, k-2\}$, then 
	\begin{equation}\label{eq higherder quotient 1} 
		|\partial^\lambda\big(v/d^{\mu+1}\big)\big| \leq C_kd^{\beta-2-|\lambda|}\quad \textrm{in}\quad \Omega\cap B_\frac{1}{2}.
	\end{equation}
for every $|\lambda| \in  \{\lfloor \beta \rfloor-1,\cdots, k\}$.
\end{corollary}	
\begin{proof}
	If $\mu= \theta_+$, we can apply directly Theorem \ref{thm schauder} to conclude that $v= h d^{\theta_+}$ for some $h \in C^{\beta-1}$. On top of this, thanks to the boundary condition \eqref{bdry-cond higher}, we have that $h|_{\partial\Omega}\equiv0$ implying that $h/d \in  C^{\beta-2}$, proving the result in this case.\\
	
	If $\mu = \theta_-$,  we exploit the fact that $\theta_+ >\theta_-$ to apply Theorem \ref{thm higher reg linear} which combined with the boundary condition \eqref{bdry-cond higher} implies that \eqref{eq linear higher order}  holds with $Q_{1,x_0}=0$ for any $x_0$. The rest of the argument thus follow proceeding as in the proof of Theorem \ref{thm schauder}.\\
	
	Assuming now $v\in C^k(\Omega\cap B_1)$ and $f\in C^k(\Omega\cap B_1)$ for some $k>\beta-2$ together with the condition \eqref{eq decaying f}, we notice that as in the proof of  Theorem \ref{thm schauder},  we have that for every $z \in \partial \Omega \cap B_1$, there exists a polynomial $Q_{z} \in \mathbf{P}_{\lfloor \beta \rfloor-2}$ such that
	\begin{eqnarray}\label{eq approx higher mu+1} 
		\Big|v(x)-Q_{z}(x)d^{\mu+1}(x)\Big|\leq C |x-z|^{{{\mu+1} + \beta-2}}\quad  \mbox{for $x\in B_{1}(z)\cap \Omega$}.
	\end{eqnarray}		

So, as in the proof of Theorem \ref{thm schauder}, for any  $y_0 \in  \Omega\cap B_\frac{1}{2}(z)$ such that $\text{dist}(y_0)= |y_0-z|=2 r$ we define the rescaling
	\begin{equation*}\label{eq rescaling+1}
		v_r(x) = r^{{{-(\mu+1)} - \beta+2}}v(y_0+rx)-  r^{2 - \beta}Q_{z}(y_0+rx)b_r(x)^{\mu+1},
	\end{equation*}
	with  $b_r(x)= \frac{d(y_0+rx)}{r}$. Since we are assuming $\beta>2$, we can argue as in the proof of Theorem \ref{thm schauder} to show that
		\begin{equation}\label{eq rescaled operator 3}
		\frac{-1}{b_r(x)^{\mu-1}}\Delta v_r+\frac{a_r(x)}{b_r(x)^{\mu-1}} v_r=  r^{2 - \beta}\Big(f_0(y_0+rx)+g(y_0+rx)\Big),
	\end{equation}
with $a_r(x)=\frac{\kappa}{b_r(x)^2}$ and where $g \in C^{\beta-2}(\overline{\Omega}\cap B_1)$ is given by the identity 	$L_\kappa(Q_{z} d^{\mu}) = -g(z) d^{\mu-1}$ and thanks to \eqref{eq der distance}, satisfies
\begin{equation}\label{eq bry control g}
	|\partial^{\alpha} g| \leq C_{\alpha} d^{\beta-2 -|\alpha|}, \quad \mbox{ in $\Omega\cap B_1$},
\end{equation}
for $|\alpha|>\beta-2$. Thus, by applying interior Schauder estimates to \eqref{eq rescaled operator 3}, we have that for any multi-index $\lambda$ with $|\lambda| \in  \{\lfloor \beta \rfloor-1,\cdots, k\}$
	\begin{eqnarray*}
		\Vert \partial^\lambda v_r \Vert_{L^\infty(B_\frac{1}{2})}\leq C\Big( \Vert v_r\Vert_{L^\infty(B_1)}+ r^{2-\beta} 	\Vert D^{k-2}(f(y_0+r \cdot)+g(y_0+r\cdot))\Vert_{L^\infty(B_1)}\Big)
	\end{eqnarray*}
	which by \eqref{eq decaying f}, after rescaling and combining \eqref{eq decaying f} with \eqref{eq bry control g} , implies
	\begin{eqnarray*}\label{eq int schauder}
		 r^{{{-(\mu+1)} - \beta+2+k}}\Vert \partial^\lambda v \Vert_{L^\infty(B_\frac{r}{2})}&\leq& C\Big(1+  r^{k-\beta}\Vert D^{k-2}(f(y_0+r \cdot)+g(y_0+r\cdot))\Vert_{L^\infty(B_1)}\Big)\leq C.
	\end{eqnarray*}

This shows that the hypotheses of Lemma \ref{lemma quotient} are fulfilled from where \eqref{eq higherder quotient 1} follows.	
\end{proof}

\section{Higher order free boundary regularity}
\label{sec5}

This section is devoted to the proof of  Theorem \ref{thm 1}. For this, we first show step (a) in \eqref{strategy-proof}.

\begin{theorem}\label{thm higher reg solutions}
	Let $\gamma \in (0,2)$ and let $\Omega\subset\R^n$ be a $C^\beta$ domain, with $\beta>1$ and $\beta\notin \mathbb N$.  Let  $0\leq u \in C^{\frac{2}{2-\gamma}}$ be a nontrivial solution of
	\begin{equation}\label{eq-alt-phil3}
		\begin{split}
			\Delta u & = u^{\gamma-1} \quad \textrm{in}\quad \Omega\cap B_1 \\
			u& = 0  \qquad \  \textrm{on}\quad \partial\Omega\cap B_1.
		\end{split}
	\end{equation}
	Then, we have 
	\[u/d^{\frac{2}{2-\gamma}} \in C^{\beta-1}(\overline \Omega\cap B_r)\qquad\textrm{and}\qquad u_i/d^{\frac{\gamma}{2-\gamma}}\in C^{\beta-1}(\overline \Omega\cap B_r) \]
	for any $r<1$.
\end{theorem}
\begin{proof}
	
The proof is by induction on $\lfloor \beta\rfloor$ and it is done through steps.\\

\medskip

\noindent \textit{Step 1:} We find that $v:=  u - c_\gamma d^{\frac{2}{2-\gamma}}$ satisfies
\begin{equation}\label{eq for v}
	L_\kappa v = h d^{\frac{2}{2-\gamma}-2},
\end{equation}
for some function $h$ whose regularity will be improved along the proof.\\

Computing directly, we have that
\begin{equation}\label{lkjh4}
	\Delta v = u^{\gamma-1} - \big(c_\gamma d^{\frac{2}{2-\gamma}}\big)^{\gamma-1} + gd^{\frac{2}{2-\gamma}-2},
\end{equation}
where $c_\gamma$ is given by \eqref{eq coefficient distance}, and $g$ is given by the identity
\begin{equation}\label{eq def g}
	gd^{\frac{2}{2-\gamma}-2} = \Delta\big(c_\gamma d^{\frac{2}{2-\gamma}}\big) - \big(c_\gamma d^{\frac{2}{2-\gamma}}\big)^{\gamma-1},
\end{equation}

Now, it follows from Taylor's Theorem (applied to the function $t^{\gamma-1}-1$ at $t=1$) that 
\[p^{\gamma-1}-q^{\gamma-1}=(\gamma-1)(p-q)q^{\gamma-2} + (p/q-1)^2 q^{\gamma-1} \int_0^1 (\gamma-1)(\gamma-2) \big(1+t p/q\big)^{\gamma-3}(1-t) dt \]
for all positive numbers $p,q>0$.
Hence, we deduce that 
\begin{equation}\label{lkjh5}
	u^{\gamma-1} - \big(c_\gamma d^{\frac{2}{2-\gamma}}\big)^{\gamma-1}=\kappa\frac{v}{d^2} + f d^{\frac{2}{2-\gamma}-2} 
\end{equation}
where $\kappa$ is given by \eqref{eq aux constant} and
\begin{equation}\label{eq formula f}
	f:= \big(u/d^{\frac{2}{2-\gamma}}-c_\gamma\big)^2 \int_0^1 (\gamma-1)(\gamma-2) \big(c_\gamma+t u/d^{\frac{2}{2-\gamma}}\big)^{\gamma-3}(1-t) dt.
\end{equation}

Combining \eqref{lkjh4} and \eqref{lkjh5} we deduce
\begin{equation}\label{lkjh6}
	L_\kappa v = f d^{\frac{2}{2-\gamma}-2}  + g d^{\frac{2}{2-\gamma}-2}.
\end{equation}

Notice that $f$ has at least the same regularity up to the boundary as $u/d^{\frac{2}{2-\gamma}}$.\\

\medskip

\noindent {\it Step 2:} We prove the base case. Moreover, we show that for any $\beta \in (1,2)$ we have that
\begin{equation}\label{lkjh}
	u/d^{\frac{2}{2-\gamma}} \in C^{\beta-1}\qquad\textrm{and}\qquad u_i/d^{\frac{\gamma}{2-\gamma}}\in C^{\beta-1} \qquad  \textrm{for all}
\end{equation}
for $i=1,\cdots n.$\\

In light of Proposition \ref{prop regular solutions}, it suffices to show $u_i/d^{\frac{\gamma}{2-\gamma}}\in C^{\beta-1}$. On the other hand, we claim that
\begin{equation}\label{eq initial boundary decay}
		|D^1\big(u/d^{\frac{2}{2-\gamma}}\big)\big| \leq Cd^{\beta-2}\quad \textrm{in}\quad \Omega\cap B_r.
\end{equation}
for $r\in (0,1)$.\\

Indeed, arguing as in Corollary \ref{corollary schauder},  given any $z \in \partial \Omega \cap B_1$, for any  $y_0 \in B_\frac{1}{2}\cap \Omega$ such that $\text{dist}(y_0)= |x_0-z|=2 r$ we can consider the rescaling
\begin{equation*}\label{eq rescaling+2}
	v_r(x) = r^{-\beta+1-\frac{2}{2-\gamma}}v(y_0+rx)
\end{equation*}
which is bounded in $B_1$ in virtue of the regularity of $u/d^\frac{2}{2-\gamma}$. Hence, we can rewrite \eqref{lkjh6} as
\begin{equation}\label{eq rescaled operator 4}
	\frac{-1}{b_r(x)^{\frac{2}{2-\gamma}-2}}\Delta v_r+\frac{a_r(x)}{b_r(x)^{\frac{2}{2-\gamma}-2}} v_r=  r^{2 - \beta}\Big(f(y_0+rx)+g(y_0+rx)\Big),
\end{equation}
with $b_r(x)= \frac{d(y_0+rx)}{r}$ and $a_r(x)=\frac{\kappa}{b_r(x)^2}$. Let notice now that since $|g| \leq  \in C d^{\beta-1}$ thanks to Lemma \ref{lemma generalized distance} and $|f|\leq C d^{2(\beta-1)}$ we have that both functions are bounded. Thus, we can use standard gradient estimates to deduce
\begin{eqnarray*}
	\Vert \nabla v_r \Vert_{L^\infty(B_\frac{1}{2})}\leq C\Big( \Vert v_r\Vert_{L^\infty(B_1)}+ r^{2-\beta} 	\Vert f(y_0+r \cdot)+g(y_0+r\cdot))\Vert_{L^\infty(B_1)}\Big)\leq C
\end{eqnarray*}
implying
\begin{eqnarray*}
	\Vert \nabla v \Vert_{L^\infty(B_\frac{r}{2})}\leq r^{\frac{2}{2-\gamma}+\beta-2}.
\end{eqnarray*}
From here, we can deduce \eqref{eq initial boundary decay} from Lemma \ref{lemma quotient}.\\

With \eqref{eq initial boundary decay} at hand, we notice that the elementary identity
\[u_i/d^{\frac{\gamma}{2-\gamma}} = d\partial_i\left(u/{d^{\frac{2}{2-\gamma}}}\right)+cu/{d^{\frac{2}{2-\gamma}}},\]
yields 
\[u_i/d^{\frac{\gamma}{2-\gamma}} \in C^{\beta-1}(\overline \Omega\cap B_r),\]
completing the base case.\\

We shall assume from now on $\beta>2$ and, by inductive hypothesis
\begin{equation}\label{eq inductive}
	u/d^{\frac{2}{2-\gamma}} \in C^{\beta-2}(\overline \Omega\cap B_1).
\end{equation}
In particular, from \eqref{eq formula f} we also have that $f \in C^{\beta-2}$; whereas 
\begin{equation}\label{lkjh2}
	g\in C^{\beta-1}(\overline\Omega\cap B_1),
\end{equation}
thanks to Lemma \ref{lemma generalized distance}.\\

\medskip

\noindent \textit{Step 3:} We claim that 
\begin{equation}\label{eq higherder quotient}
	|D^k\big(u/d^{\frac{2}{2-\gamma}}\big)\big| \leq C_kd^{\beta-2-k}\quad \textrm{in}\quad \Omega\cap B_r,
\end{equation}
for all $k>\beta-2$. \\

The proof of this claim will follow from an inductive application of Corollary \ref{corollary schauder}. Let us start notice that the base case when $k =\lfloor\beta \rfloor-1$ follows immediately from the boundedness of $f$ and $g$ arguing as in the previous step, i.e., we have that 
\begin{equation*}
	|D^1\big(u/d^{\frac{2}{2-\gamma}}\big)\big| \leq Cd^{\beta-3}\quad \textrm{in}\quad \Omega\cap B_r.
\end{equation*}
for $r\in (0,1)$.\\

Let us assume now that \eqref{eq higherder quotient} holds for $k \leq l$ and let is show its validity for $l+1$. Since $g$ is defined by the identity \eqref{eq def g}, Lemma \ref{lemma generalized distance} implies that 
\begin{equation}\label{eq bry control g 2}
	|\partial^{\alpha} g| \leq C_{\alpha} d^{\beta-1 -|\alpha|}, \quad \mbox{ in $\Omega\cap B_1$},
\end{equation}
for $|\alpha|>\beta-1$. On the other hand, from \eqref{eq formula f} we have that $f = H\Big( u/d^\frac{2}{2-\gamma}\Big)$ with $H \in C^\infty$ -bearing in mind also that $ u/d^\frac{2}{2-\gamma}$ is uniformly bounded from below. So, the inductive hypotheses implies that for $|\alpha| \in \{ \lfloor \beta\rfloor-1, \cdots, l\}$
\begin{equation}\label{eq bry control f}
	|\partial^{\alpha} f| \leq C_{\alpha} d^{\beta-2 -|\alpha|}, \quad \mbox{ in $\Omega\cap B_1$}.
\end{equation}

Hence, \eqref{eq bry control g 2} and \eqref{eq bry control f} together with \eqref{eq inductive} allows us to invoke Corollary \ref{corollary schauder} and show that \eqref{eq higherder quotient} holds with $k=l+1$, completing the inductive step.
\medskip

\noindent \textit{Step 4:} Our next goal is to show that $f \in C^{\beta-1}(\overline \Omega\cap B_r)$ for all $r<1$.
We first prove it at all boundary points on $\partial\Omega$.
Namely, for any $z\in \partial\Omega \cap B_r$. Let us start observing that by inductive hypothesis we have
\[u/d^{\frac{2}{2-\gamma}} - c_\gamma = P_z(x) + O(|x-z|^{\beta-2}),\]
for all $x\in \Omega\cap B_1$, for some polynomial $P_z$ satisfying $P_z(z)=0$, and where the (implicit) constants are independent of $z\in \partial\Omega \cap B_r$.
Now, since $P_z(x)=O(|x-z|)$ we obtain
\[\left(P_z+O(|x-z|^{\beta-2})\right)^2 = P_z^2 + O(|x-z|^{\beta-1}),\]
provided that $\beta>3$, and therefore
\[\big(u/d^{\frac{2}{2-\gamma}} - c_\gamma\big)^2 = P_z^2 + O(|x-z|^{\beta-1}).\]
When $\beta<3$ we use \ Proposition \ref{prop regular solutions} instead, to get
\[u/d^{\frac{2}{2-\gamma}} - c_\gamma = O(|x-z|^{\alpha})\]
and choosing $2\alpha=\beta-1$ we deduce
\[\big(u/d^{\frac{2}{2-\gamma}} - c_\gamma\big)^2 = O(|x-z|^{\beta-1}).\]
In both cases, we also have that
\[\int_0^1 (\gamma-1)(\gamma-2) \big(c_\gamma+t u/d^{\frac{2}{2-\gamma}}\big)^{\gamma-3}(1-t) dt \in C^{\beta-2},\]
and therefore 
\[f(x) = Q_z(x) + O(|x-z|^{\beta-1})\]
for all $x\in \Omega\cap B_1$, for some polynomial $Q_z$ satisfying $Q_z(z)=0$, and where the (implicit) constants are independent of $z\in \partial\Omega \cap B_r$.

This means that $f$ is pointwise $C^{\beta-1}$ on $\partial\Omega\cap B_r$.
We now want to combine this information with interior regularity estimates, to deduce that $f \in C^{\beta-1}(\overline \Omega\cap B_r)$.

Indeed, thanks to step 3 and the Leibniz Rule, choosing $k=\lfloor \beta\rfloor$, and using that $\big|u/d^{\frac{2}{2-\gamma}}-c_\gamma\big|\leq Cd$, we have 
\[\begin{split}
\big|D^k\big(u/d^{\frac{2}{2-\gamma}}-c_\gamma\big)^2\big| & \leq 
C\sum_{i=0}^{k/2} \big|D^i\big(u/d^{\frac{2}{2-\gamma}}-c_\gamma\big)\big|\cdot \big|D^{k-i}\big(u/d^{\frac{2}{2-\gamma}}-c_\gamma\big)\big|  \\
&\leq Cd^{\beta-2-k}d + C\sum_{i=1}^{k/2} \big|D^i\big(u/d^{\frac{2}{2-\gamma}}\big)\big|\cdot \big|D^{k-i}\big(u/d^{\frac{2}{2-\gamma}}\big)\big| \\
&\leq Cd^{\beta-1-k} + C\sum_{i=1}^{k/2} (1+d^{\beta-2-k+i}) \\
&\leq Cd^{\beta-1-k} + Cd^{\beta-1-k}.
\end{split}\]
Applying again the Leibniz Rule, we find 
\[\big|D^kf\big|\leq Cd^{\beta-1-k}\quad \textrm{in}\quad \Omega\cap B_r,\]
and in particular  
\[f \in C^{\beta-1}(\overline \Omega\cap B_r),\]
as wanted.\\

\medskip

\noindent {\it Step 5:}  Our next step consists in running a blow-up argument as in the proof of Theorem \ref{thm higher reg linear} and show the following:\\

	\textit{Claim:} There exists a constant $C>0$ and a radius $\rho>0$, such that for every $x_0 \in \partial \Omega \cap B_\rho$, there exists a polynomial $Q_{x_0} \in \mathbf{P}_{\lfloor \beta \rfloor -1}$ such that
\begin{eqnarray}\label{eq linear solution} 
	\Big|v(x)-Q_{1,x_0}(x)d^\frac{2}{2-\gamma}(x)\Big|\leq C |x-x_0|^{\frac{2}{2-\gamma} + \beta - 1}\quad  \mbox{for $x\in B_{\rho}(x_0)\cap \Omega$}.
\end{eqnarray}		

\noindent {\it Proof of the claim:} This proof follows the same lines of the proof of Theorem \ref{thm higher reg linear}, for this reason we just sketch the common features of the argument, stressing only on the (slight) variations of our reasoning.\\

Let us assume $x_0=0$ and argue by contradiction assuming a sequence of domains $\Omega_j$ with $\Vert \Omega_j\Vert_{C^{\beta}}\leq 1$ and, in consequence, of  distance functions $d_j$ such that for any sequence of  polynomials $P_{j} \in \mathbf{P}_{\lfloor \beta \rfloor -1}$ 
\begin{equation}\label{eq contradiction solution}
	\sup_{r >0} \sup_{j\in \N} \frac{1}{r^{\frac{2}{2-\gamma} + \beta-1}} \Big\Vert v_j-P_{j} d_j^\frac{2}{2-\gamma}\Big\Vert_{L^\infty(B_r)}=\infty,
\end{equation}
with
\begin{equation}\label{eq seq solution BVP}
	L_\kappa(v_j)= h_{j}d_j^{\frac{2}{2-\gamma}-2},\quad \text{in}\,\, B_2\cap \Omega
\end{equation}
with $\Vert h_{j}\Vert_{C^{\beta -1}( \overline{\Omega_j}\cap B_1)}\leq 1$  for $j \in \N$.\\

We select the least squares polynomial $Q_{j,r} \in \mathbf{P}_{\lfloor \beta \rfloor-1}$ characterized by the orthogonality condition
\begin{equation}\label{eq orthogonality solution}
	\int_{B_r} \Big(v_j -Q_{1,j,r} d_j^\frac{2}{2-\gamma}\Big)P_1(x)d_j^\frac{2}{2-\gamma}=0,
\end{equation}
for any $P_1 \in  \mathbf{P}_{\lfloor \beta \rfloor -1}$ for $r>0$ and $j\in \N$.\\

Let us define
$$\theta(r) =  \sup_{\rho >r }\sup_{j\in \N} \frac{1}{r^{{\frac{2}{2-\gamma} + \beta-1}}} \Big\Vert v_j -Q_{j,r} d_j^\frac{2}{2-\gamma}\Big\Vert_{L^\infty(B_r)}.$$

Arguing as in Lemma \ref{eq unboundedness of theta}, we have that $\lim_{r\to 0^+} \theta(r)=\infty$. Thus, we can pick a sequence $\{r_j\}_{j\in \N}$  with $r_j\to 0^+$ as $j\to \infty$ such that 	
\begin{equation}\label{eq blowup sequence soltuion}
	w_{j}(x):= \frac{ v_{j}(r_j x)-Q_{j,r_j}(r_jx)d_j^\frac{2}{2-\gamma}(r_j x)}{r_j^{{\frac{2}{2-\gamma}  +\beta-1}}\theta(r_j)},
\end{equation}
satisfies $ \Vert w_j\Vert_{L^\infty(B_1)} \in [\frac{1}{2},1]$ for $j\in \N$.\\		

Additionally, the same reasoning of Lemma \ref{lemma growth property} provides the growth bound
\begin{equation}\label{eq growthcond soltuion}
	\Vert w_j \Vert_{L^\infty(B_R)}\leq C R^{\frac{2}{2-\gamma}+\beta-1}.
\end{equation}	

One of the main differences with Theorem \ref{thm higher reg linear} appears when we compute the Laplacian of the auxiliary term $Q_{1,j,r} d_j^\frac{2}{2-\gamma}$. More precisely, using \eqref{eq laplacian altphillips} to deduce that
$$\Delta d_j^\frac{2}{2-\gamma}(x)= d_j^\frac{2}{2-\gamma}(x)(R_j(x)d_j(x)+c_\gamma^{\gamma-2}+d_j(x)E(x)),$$
with $|E(x)|\leq C|x|^{\beta-2}$ and $R_j \in  \mathbf{P}_{\lfloor \beta \rfloor-1}$ with $\Vert R_j \Vert_{L^\infty(B_1)}\leq C_\beta$,
we can argue as in Lemma \ref{lemma growth property} to obtain the expansion
\begin{eqnarray}\label{eq first order solution}
	\frac{1}{\theta(r) r^{\frac{2}{2-\gamma}+\beta-1}}\Delta  \Big(Q_{r}(rx)d_j^\frac{2}{2-\gamma}(r x)\Big)&=&\frac{1}{\theta(r) r^{\frac{2}{2-\gamma}+\beta-1}}\frac{\kappa r^2}{d_j^2(rx)}\Big(Q_{1,r}(rx)d_j^\frac{2}{2-\gamma}(r x)\Big)\\ 
	&&+P_{r}(x)\Big( \frac{d_j(rx)}{r}\Big)^{\frac{2}{2-\gamma}-2}+m_*(r,x)
\end{eqnarray}
where $P_{r} \in \mathbf{P}_{M}$ for some $M \in \N$ depending only on $\beta$ and $\kappa$ and where $|m_*(x,r)|\leq Cm(r)$ with $m$ only depending on $\beta$ and $\gamma$ and such that $m(r)\to 0$ as $r\to 0^+$.\\

By combining \eqref{eq seq linearized BVP}, \eqref{eq first order error linear},  and exploiting the uniform regularity of the functions $h_{j}$ we deduce	
\begin{eqnarray}\label{eq almost solution}
	L_\kappa ( w_j) &=& \Big( \tilde{Q}_{j}(x)+g_{j}(x) \Big)\Big( \frac{d_j(rx)}{r}\Big)^{\frac{2}{2-\gamma}-2}+h_{j}(x),
\end{eqnarray}
where $\tilde{Q}_{j} \in  \mathbf{P}_{M }$ for some $M\in \N$ and $\Vert g_{j} \Vert_{L^\infty(B_2\cap \overline{\Omega_j})}+\Vert h_{j} \Vert_{L^\infty(B_2\cap \overline{\Omega_j})}\to 0$ as $j\to \infty$. Additionally, in virtue of Lemma \ref{lemma uni approx poly}, the sequence of polynomials $\tilde{Q}_{j}$ are uniformly bounded in $j$. \\

We consider two cases: if $\gamma \geq \frac{2}{3}$, then $\frac{2}{2-\gamma} = \theta_+ +1$; whereas, if $\gamma < \frac{2}{3}$, we have that  $\frac{2}{2-\gamma} = \theta_- +1$. In both cases, we invoke Corollary \ref{corollary schauder} to obtain uniform boundary estimates that guarantee equicontinuity for $w_j/d_j^{\frac{2}{2-\gamma}}$. All in all, the solutions $\{w_j\}_{j\in \N}$ to the sequence of problems \eqref{eq almost solution linear lower} are equicontinuous and, therefore, satisfy that $w_j \to w_0$ locally uniformly. So, thanks to  Lemma \ref{lemma uni approx poly},  and \eqref{eq growthcond der low} we deduce that			
\begin{equation}\label{eq limiting equation solution}
	\begin{cases}
		L_\kappa w_0 = Q(x)(x_n)_+^{\frac{2}{2-\gamma}-2}, \quad \{x_n>0\}\\
		\Vert w_0\Vert_{L^\infty(B_R\cap \R^n_+)} \leq CR^{{\frac{2}{2-\gamma}}+\beta-1}
	\end{cases}
\end{equation}
with $Q \in  \mathbf{P}_{M}$ for some $M\in \N$. Additionally, thanks to Corollary \ref{corollary schauder} we have that 
\begin{equation}\label{eq boundary lower gamma}
	|w_0|\leq C x_n^{\frac{2}{2-\gamma}}.
\end{equation}

So, since $w_0$ satisfies \eqref{eq limiting equation solution} and  \eqref{eq boundary lower gamma}, we can apply Lemma  \ref{lemma general Liouville} to conclude that
$$w_0= P_1(x)(x_n)_+^\frac{2}{2-\gamma}$$ 
with $P_1 \in \mathbf{P}_{\lfloor \beta\rfloor-1}$. However, taking limits in the orthogonality condition \eqref{eq orthogonality solution} as $j\to \infty$, yields	
\begin{equation*}
	\int_{B_1} w_0  Q_1(x)x_n^\frac{2}{2-\gamma}=0,
\end{equation*}
for all  $Q_1 \in \mathbf{P}_{\lfloor \beta\rfloor-1}$. Thus, $w_0=0$ contradicting $\Vert w_0\Vert_{L^\infty(B_1)}\geq \frac{1}{2}.$\\	

\medskip

\noindent {\it Step 6:} We conclude the induction argument.\\

Once we know that \eqref{eq linear solution} holds, it follows arguing as in the proof of Theorem \ref{thm schauder} that
\[u/d^{\frac{2}{2-\gamma}} \in C^{\beta-1}(\overline \Omega\cap B_r).\]
Moreover, exactly as in step 3, we now have
\[\big|D^k\big(u/d^{\frac{2}{2-\gamma}}\big)\big| \leq Cd^{\beta-1-k}\quad \textrm{in}\quad \Omega\cap B_r,\]
for all $k>\beta-1$, which combined with
\[u_i/d^{\frac{\gamma}{2-\gamma}} = d\partial_i\left(u/{d^{\frac{2}{2-\gamma}}}\right)+cu/{d^{\frac{2}{2-\gamma}}},\]
yields 
\[\big|D^k\big(u_i/d^{\frac{\gamma}{2-\gamma}}\big)\big| \leq Cd^{\beta-1-k}\quad \textrm{in}\quad \Omega\cap B_r.\]
This means that 
\[u_i/d^{\frac{\gamma}{2-\gamma}} \in C^{\beta-1}(\overline \Omega\cap B_r),\]
and we are done.

\end{proof}

We can finally give the:

\begin{proof}[Proof of  Theorem \ref{thm 1}]
Assume that $x_0=0$ and that the inner unit normal vector at $0$ is $\nu=e_n$, and let $\Omega=\{u>0\}$.
Since $\Delta u_i = (\gamma-1)u^{\gamma-2} u_i$ in $\Omega$ for all $i=1,...,n$, then it follows that 
\[w:=\frac{u_i}{u_n}\]
solves
\[{\rm div}(u_n^2\nabla w) = 0 \quad \textrm{in}\quad \Omega\cap B_1.\]
Thanks to Proposition, we have that 
\[u_n^2\asymp d^s \qquad  \textrm{and} \qquad u_n^2/d^s \in C^\alpha(B_{r_\circ}\cap \overline\Omega) \qquad \textrm{for any}\quad \alpha\in(0,1),\]
where $r_\circ>0$ and
\[s:={\frac{2\gamma}{2-\gamma}}>0.\]
Hence, we can write the PDE for $w$ as 
\begin{equation}\label{werty} 
{\rm div}\Big(d^s a(x)\nabla w \Big) = 0 \quad \textrm{in}\quad \Omega\cap B_{r_\circ},
\end{equation}
with $a(x)\in C^\alpha(B_{1/2}\cap \overline\Omega)$ and $\lambda\leq a(x)\leq \Lambda$ for some constants $\lambda,\Lambda>0$.

Let us see now what the boundary condition for $w$ on $\partial\Omega$ is.
Again by Propositions \ref{prop regular solutions}, we have that 
\[w=\frac{u_i}{u_n} \in C^{\alpha}(B_{r_\circ}\cap \overline\Omega)\quad \textrm{for any}\quad \alpha\in(0,1),\]
and combining this with interior estimates, a standard argument yields $|\nabla w|\leq Cd^{\alpha-1}$ in $\Omega\cap B_{r_\circ}$.
In particular, taking $\alpha>1-s$, we deduce that
\[\lim_{t\downarrow0} \frac{w(z+t\nu)-w(z)}{t^{1-s}} = 0 \qquad \textrm{for all}\quad z\in \partial\Omega\cap B_{r_\circ}\]
or, equivalently,
\begin{equation}\label{werty2} 
\lim_{t\downarrow0} t^s \,\nabla w(z+t\nu) \cdot \nu = 0 \qquad \textrm{for all}\quad z\in \partial\Omega\cap B_{r_\circ}.
\end{equation}
This means that $w$ is a weak solution of \eqref{werty} with Neumann-type boundary conditions \eqref{werty2}, and therefore it follows from \cite[Theorem~1.1]{TTV22} that 
\[w\in C^{1+\alpha}(B_r\cap \overline\Omega) \quad \textrm{for any}\quad r\in(0,r_\circ).\]
But then, exactly as in \eqref{strategy-obstacle}, we deduce that $\nu\in C^{1,\alpha}$ and $\partial\Omega \in C^{2,\alpha}$.

Now, since $\Omega$ is $C^{2,\alpha}$, we can apply Theorem \ref{thm higher reg solutions} to find that
\[u_i/d^{\frac{\gamma}{2-\gamma}} \in C^{1,\alpha}(\overline\Omega \cap B_r ),\]
and in particular
\[a(x)= u_n^2/d^s \in C^{1,\alpha}(\overline\Omega \cap B_r).\]
Hence, we can apply again  \cite[Theorem~1.1]{TTV22} to deduce that 
\[w\in C^{2+\alpha}(\overline\Omega \cap B_r ) \quad \textrm{for any}\quad r\in(0,r_\circ).\]
Iterating this procedure, we find that 
\[\partial\Omega \in C^\infty \qquad \textrm{and}\qquad  u_i/d^{\frac{\gamma}{2-\gamma}} \in C^\infty.\]
This yields $u/d^{\frac{2}{2-\gamma}} \in C^\infty$ and thus $u^\frac{2-\gamma}{2}= \Big(u/d^{\frac{2}{2-\gamma}}\Big)^\frac{2-\gamma}{2} d  \in C^\infty$ , and the theorem is proved.
\end{proof}

\section{Polynomial lemmas}
\label{sec6}

In this section we gather technical lemmas that provide control over the polynomials used in the least squares approximation in our various blow-up arguments. All of these results can be seen as a generalization to those in \cite{AR} regarding the control of their least square polynomials, in the sense that we deal with more than one element in the kernel of the limiting linearized operator.\\

\begin{lemma}\label{lemma second polynomial scalinginv bounds}
	Let $\kappa \geq  - \frac{1}{4}$ and $\theta_\pm$ as in \eqref{eq powers}.	Let $k_1, k_2 \in \N $, $\beta >1$, and let $\Omega$ be such that $\Vert\partial \Omega \cap B_2\Vert_{C^{\beta}}\leq 1$. There exists $\rho>0$ such that for any $r\in (0,\rho)$ and for any $Q_i \in \mathbf{P}_{k_i}$ for $i=1,2$, if	
\begin{equation}\label{eq linfty bound}
		\Vert Q_1 d^{\theta_+}+Q_2 d^{\theta_- + 1}\Vert_{L^\infty(\partial\Omega \cap B_r )} \leq \theta,
\end{equation}
 where $\theta>0$ and with $Q_2=0$ if either $\theta_- \leq0$ or if $\kappa = - \frac{1}{4}$, then	
	\begin{equation}\label{eq rescaled ineq}
		\Vert Q_1\Vert_{L^\infty(\partial\Omega \cap B_{r})}r^{\theta_+}+ 	\Vert Q_2\Vert_{L^\infty(\partial\Omega \cap B_r)}r^{\theta_-+ 1}\leq C\theta,
	\end{equation}
and
\begin{equation}\label{eq coefficients bounds}
	\sum_{|\mu|\leq k_1}r^{|\mu|+\theta_+} \big|q^{(\mu)}_1\big|+ \sum_{|\mu|\leq k_2}r^{|\mu|+\theta_- +1} \big|q^{(\mu)}_2\big| \leq C\theta,
\end{equation}
	with $C=C(\rho,n,k_1,k_2, \kappa)$.	
\end{lemma}
\begin{proof}	
Let us notice first that by the equivalence of norms in finite dimensional spaces, there exists a constant $C=C(k_1,k_2,\kappa)$ such that	
	\begin{eqnarray}\nonumber
	\sum_{|\mu|\leq k_1} \big|p^{(\mu)}_1\big|+ \sum_{|\mu|\leq k_2}\big|p^{(\mu)}_2\big|&\leq& C(\Vert P_1(x) \Vert_{L^\infty(\R^n_+\cap B_1)}+\Vert P_2(x)\Vert_{L^\infty(\R^n_+\cap B_1)})\\\nonumber
	&\leq&  C\Big(	\Vert P_1(x) x_n^{\theta_+} \Vert_{L^\infty(\R^n_+\cap B_1)}+\Vert P_2(x)x_n^{\theta_-+1} \Vert_{L^\infty(\R^n_+\cap B_1)}\Big)\\\label{eq polyconv dist}
	&\leq& C 	\Vert P_1(x) x_n^{\theta_+}+P_2(x)x_n^{\theta_-1+1}\Vert_{L^\infty(\R^n_+\cap B_1)},
	\end{eqnarray}	
where the last inequality uses fundamentally the linear independence of $x_n^{\theta_+}$ and $x_n^{\theta_-+1}$, i.e., $\theta_+-\theta_- \notin \N$.\\
	
Let $F_r := \frac{1}{r} \Omega$. By compactness, we can deduce from \eqref{eq polyconv dist} and from the $C^{\beta}$ bound in $\partial \Omega \cap B_2$, we can find $\rho>0$ such that for $r\in (0,\rho)$ and for any $P_i \in  \mathbf{P}_{k_i}$,

	\begin{eqnarray}\nonumber
	\sum_{|\mu|\leq k_1} \big|p^{(\mu)}_1\big|+ \sum_{|\mu|\leq k_2} \big|p^{(\mu)}_2\big|
	&\leq& \Big\Vert P_1(x)  \Big\Vert_{L^\infty(F_r\cap B_1)}+\Vert P_2(x) \Vert_{L^\infty(F_r \cap B_1)}\\\label{eq polyconv dist approx}
	&\leq& C 	\Big\Vert P_1(x) \Big(\frac{d(rx)}{r}\Big)^{\theta_+}+P_2(x)\Big(\frac{d(rx)}{r}\Big)^{\theta_-+1} \Big\Vert_{L^\infty(F_r\cap B_1)}.
\end{eqnarray}

Let $Q_i(x) \in \mathbf{P}_{k_i}$ for $i=1,2$ and let $r\in (0,r_0)$. Let us consider the polynomials
$$P_1(x) =\sum_{ |\mu|\leq k_1} r^{|\mu|+\theta_+}  q_1^{(\mu)} x^\alpha, \qquad P_2(x) =\sum_{ |\mu|\leq k_2} r^{|\mu|+\theta_-+1}  q_2^{(\mu)} x^\alpha.$$
 With this choosing of $P_1$ and $P_2$ in \eqref{eq polyconv dist approx}, we deduce, after a change of variables, that
\begin{eqnarray*}
\sum_{|\mu|\leq k_1}r^{|\mu|+\theta_+} \big|q^{(\mu)}_1\big|+ \sum_{|\mu|\leq k_2}r^{|\mu|+\theta_-+1} \big|q^{(\mu)}_2\big|
&\leq& r^{\theta_+}	\Big\Vert Q_1 \Big\Vert_{L^\infty( \Omega \cap B_r)}+r^{\theta_-+1}\Vert Q_2(x) \Vert_{L^\infty( \Omega \cap B_r)}\\
&\leq& C	\Big\Vert Q_1 d^{\theta_+}+Q_2d^{{\theta_-+1}} \Big\Vert_{L^\infty( \Omega \cap B_r)},
\end{eqnarray*}
which yields \eqref{eq rescaled ineq} and \eqref{eq coefficients bounds}.
\end{proof}

\begin{lemma}\label{eq unboundedness of theta}
	Let $\kappa \geq  - \frac{1}{4}$ and $\theta_\pm$ as in \eqref{eq powers}.	Let $\beta >1$ with $\beta \notin \N$, and let $\Omega$ be such that $\Vert\partial \Omega \cap B_2\Vert_{C^{\beta}}\leq 1$. Let  $v \in L^\infty(B_2)$ and let us assume that for some $\rho>0$ and each $r\in (0,\rho)$ there exists $Q_{1,r} \in \mathbf{P}_{\lfloor \beta \rfloor -1 }$ and $Q_{2,r} \in \mathbf{P}_{\lfloor \beta +\theta_+-\theta_- \rfloor}-2$ such that 
		\begin{equation}\label{eq decayinginit}
				\Vert v-Q_{1,r}d^{\theta_+}-Q_{2,r}d^{\theta_-}\Vert_{L^\infty(B_r)} \leq  c_0  r^{\theta_++\beta-1},
		\end{equation}
	where $d$ is extended by $0$ outside of $\Omega \cap B_2$, and with $Q_{2,r}=0$ if $\theta_- \leq0$ or if $\kappa = -\frac{1}{4}$.\\
	
	Then there exists $\tilde{Q_1}  \in \mathbf{P}_{\lfloor \beta \rfloor -1 }$  and $\tilde{Q_2} \in \mathbf{P}_{\lfloor \beta +\theta_+-\theta_- \rfloor}-2$  such that 		
			\begin{equation}\label{eq decayinghompoly}
			\Vert v-\tilde{Q_1}d^{\theta_+}-\tilde{Q_2}d^{{\theta_-+1}}\Vert_{L^\infty(B_r)} \leq Cc_0 r^{\beta+\theta_+-1},
		\end{equation}
	for $r \in (0,\rho)$.	
\end{lemma}
\begin{proof}
Let us consider \eqref{eq decayinghompoly} for the values $r$ and $2r$. Subtracting them and then applying the triangular inequality yields
\begin{equation*}
	\Big\Vert (Q_{1,2r}-Q_{1,r}) d^{\theta_+}+  (Q_{2,2r}-Q_{2,r})  d^{\theta_-+1} \Big \Vert_{{L^{\infty}(B_r)}}\leq  Cc_0r^{\theta_+ +\beta-1}.
\end{equation*}

Therefore, Lemma \ref{lemma second polynomial scalinginv bounds} implies
\begin{equation}\label{eq dyadic poly1}
	|q^{(\mu_1)}_{i,r}-q^{(\mu_1)}_{i,2r}|\leq Cc_0 r^{\beta-1-|\mu_1|},
\end{equation}
with $|\mu_1|\leq \lfloor \beta \rfloor -1$; whereas
\begin{equation}\label{eq dyadic poly2}
	|q^{(\mu_2)}_{2,r}-q^{(\mu_2)}_{i,2r}|\leq Cc_0 r^{\beta-2+\theta_+-\theta_--|\mu_2|},
\end{equation}
 for  $|\mu_2|\leq \lfloor \beta +\theta_+-\theta_- \rfloor -2$. Let us remark that the latter inequality implies the uniform boundedness in $r$ for the differences $|q^{(\mu_i)}_{i,r}-q^{(\mu_i)}_{i,2r}|$, for $i=1,2$.\\

On the other hand, applying the triangular inequality in \eqref{eq decayinginit} for $r=1$ and using again Lemma \ref{lemma second polynomial scalinginv bounds}, we obtain the bound
\begin{equation}\label{eq initial bounds}
	\Vert Q_{1,1}\Vert_{L^\infty(B_1)}+	\Vert Q_{2,1}\Vert_{L^\infty(B_1)}\leq Cc_0+\Vert v\Vert_{L^\infty(B_\rho)}.
\end{equation}
The latter bound combined with \eqref{eq dyadic poly1} and \eqref{eq dyadic poly2} implies the convergence of the sequences $\{q^{(\mu)}_{i,2^{-j}}\}_{j\in \N}$ for $i=1,2$ and for each $\mu$. Denoting the corresponding limits by $q^{(\mu)}_i$, we deduce, namely for $i=2$, adding up dyadically and applying the bound that for any $r \in (0,\rho)$
\begin{eqnarray*}
	|q^{(\mu)}_{2}-q^{(\mu)}_{2,r}|&\leq& \sum_{j=0}^{\infty} |q^{(\mu)}_{2,2^{-j}r}-q^{(\mu)}_{2,2^{-j-1}r}|\leq C c_0\sum_{j=0}^{\infty} (r2^{-j})^{\beta-2+\theta_+-\theta_--|\mu|}\\
	&\leq& C c_0 r^{\beta-2+\theta_+-\theta_--|\mu|} \sum_{j=0}^{\infty} (2^{-j})^{\beta+\theta_+-\theta_--2-|\mu|}\leq C c_0 r^{\beta+\theta_+-\theta_--2-|\mu|}.
\end{eqnarray*}
The same estimate holds, mutatis mutandis, for $i=1$.\\

Let us define $\tilde{Q_1} = \sum_{|\mu|\leq \lfloor \beta  \rfloor -1} q_1^{\mu} x^\mu$ and  $\tilde{Q_2} = \sum_{|\mu|\leq \lfloor \beta +\theta_+-\theta_- \rfloor -2} q_2^{\mu} x^\mu$. Applying the triangular inequality together with the previous bound (and its analogue for $\tilde{Q_1}$) yields
\begin{eqnarray*}
	\Vert v-\tilde{Q_1}d^{\theta_+}-\tilde{Q_2}d^{{\theta_-+1}}\Vert_{L^\infty(B_r)}&\leq& \Vert v-\tilde{Q_{1,r}}d^{\gamma_1}-\tilde{Q_{2,r}}d^{{\gamma_2}}\Vert_{L^\infty(B_r)} \\
	&& 	+ \Vert (Q_{1,r}-\tilde{Q_1})d^{\theta_+}\Vert_{L^\infty(B_r)}+	\Vert (Q_{2,r}-\tilde{Q_2}) d^{\theta_-+1} \Vert_{L^\infty(B_r)}\\
	&\leq& C c_0 r^{\beta+\gamma_1},
\end{eqnarray*}
which proves the result.
\end{proof}

\begin{lemma}\label{lemma growth property}
		Let $\kappa \geq  - \frac{1}{4}$ and $\theta_\pm$ as in \eqref{eq powers}.	Let $\beta >1$ with $\beta \notin \N$, and let $\Omega$ be such that $\Vert\partial \Omega \cap B_2\Vert_{C^{\beta}}\leq 1$. Let  $v \in L^\infty(B_2)$ and let us assume that for some $\rho>0$ and each $r\in (0,\rho)$ there exists $Q_{1,r} \in \mathbf{P}_{\lfloor \beta \rfloor -1 }$ and $Q_{2,r} \in \mathbf{P}_{\lfloor \beta +\theta_+-\theta_- \rfloor-2}$ such that 
	\begin{equation}\label{eq asumption growth}
		\Vert v-Q_{1,r}d^{\theta_+}-Q_{2,r}d^{\theta_-+1}\Vert_{L^\infty(B_s)}\leq \theta(s) r^{\beta-1+\theta_+},
	\end{equation}
	where $\theta: (0,\rho)\to (0,\infty)$ is a non-increasing function, with 
	\begin{equation}
		\lim_{r\to 0^+} \theta(r) = +\infty,
	\end{equation}
	and with $Q_{2,r}=0$ if $\theta_- \leq0$ or if $\kappa = -\frac{1}{4}$.

	Then
	\begin{enumerate}
		\item 
		\begin{equation}\label{eq boundedness ply}
			 \lim_{r\to 0^+} \frac{q_{i,r}^{(\mu)}}{\theta(r)}=0,
		\end{equation}
		for $i=1,2$.
		\item For $ r\in (0, \rho)$ the functions		
		\begin{equation}\label{eq blowup sequence}
			w_{r}(x):= \frac{ v(r x)-Q_{1,r}(r x)d^{\theta_+}(r x)-Q_{2,r}(r  x)d^{\theta_-+1}(r x)}{r^{{\theta_+ +\beta-1}}\theta(r)},
		\end{equation}
		satisfy the uniform bound
		\begin{equation}\label{eq growthconclusion}
			\Vert w_r \Vert_{L^\infty(B_R)}\leq C R^{\beta-1+\theta_+},
		\end{equation}
		for $R>1$.
	
	\item Lastly, if $\beta>2$,
		\begin{eqnarray}\nonumber
		\frac{1}{\theta(r) r^{{\theta_+}+\beta-1}}\Delta  \Big(Q_{1,r}(rx)d^{\theta_+}(r x)+Q_{2,r}(rx)d^{\theta_-+1}(r x)\Big)
		&=&\frac{1}{\theta(r) r^{{\theta_+}+\beta-1}}\frac{\kappa r^2}{d^2(rx)}Q_{1,r}(rx)d^{\theta_+}(r x)\\ \notag
		&&+ \frac{1}{\theta(r) r^{{\theta_+}+\beta-1}}\frac{\kappa r^2}{d^2(rx)}Q_{2,r}(rx)d^{\theta_-+1}(r x)\\ \notag
		&&+P_{1,r}(x)\Big( \frac{d(rx)}{r}\Big)^{\theta_+-1}\\\nonumber
		&& +P_{2,r}(rx) \Big( \frac{d(rx)}{r}\Big)^{\theta_--1}\\ \label{eq laplacian linear terms} 
		&& +m_*(r,x)
	\end{eqnarray}
where $P_{i,r} \in \mathbf{P}_{M}$ for $i=1,2$ and $|m_*(x,r)|\leq Cm(r)$ such that $m(r)\to 0$ as $r\to 0^+$. Here $m$ and $M$ only depends on $\beta$, $\kappa$, and $n$. 
	\end{enumerate}
\end{lemma}
\begin{proof}	
	By subtracting \eqref{eq asumption growth} at the values  $r$ and $2r$ and then applying the triangular inequality, we obtain	
	\begin{equation*}
		\Big\Vert (Q_{1,2r}-Q_{1,r}) d^{\theta_+}+  (Q_{2,2r}-Q_{2,r}) d^{\theta_-+1} \Big \Vert_{{L^{\infty}(B_r)}}\leq  C\theta(r) r^{\beta +{\gamma_1}}
	\end{equation*}
	
	Therefore, Lemma \ref{lemma second polynomial scalinginv bounds} implies
	\begin{equation}\label{eq dyadic ineq}
		\Vert Q_{1,2r}-Q_{1,r} \Vert_{L^{\infty}(B_{2r})}r^{\theta_+}+\Vert Q_{2,2r}-Q_{2,r} \Vert_{L^{\infty}(B_{2r})} r^{\theta_-+1}\leq C \theta(r)  r^{\beta-1 +\theta_+}
	\end{equation}
	
Let $i \in \{1,2\}$. In virtue of \eqref{eq dyadic ineq}, we deduce
	\begin{equation*}
		|	q^{(\mu)}_{1,r}-q^{(\mu)}_{1,2r}|\leq C\theta(r) r^{\beta-|\mu|-1},
	\end{equation*}
	and
		\begin{equation*}
		|	q^{(\mu)}_{2,r}-q^{(\mu)}_{2,2r}|\leq C\theta(r) r^{\beta-2+\theta_+-\theta_--|\mu|}.
	\end{equation*}
	
	By using the triangular inequality, it follows that for $i=1$,
	\begin{equation}\label{eq telescopic}
		|q^{(\mu)}_{1,r}-q^{(\mu)}_{1,2^Jr}|\leq \sum_{l=0}^{J-1} |q^{(\mu)}_{1,2^{l}r}-q^{(\mu)}_{1,2^{l+1}r}|\leq C \sum_{l=0}^{J-1} \theta(r2^{l})(r2^l)^{\beta-|\mu|-1},
	\end{equation}
	where $|\mu|\leq \lfloor \beta \rfloor -1$. Reasoning similarly, we obtain
	\begin{equation}\label{eq telescopic 2}
		|q^{(\mu)}_{2,r}-q^{(\mu)}_{2,2^Jr}|\leq C \sum_{l=0}^{J-1} \theta(r2^{l})(r2^l)^{\beta-2+\theta_+-\theta_--|\mu|},
	\end{equation}
	where $|\mu|\leq \lfloor \beta+\theta_+-\theta_-\rfloor -2$.
	
	Let us notice also that  $q^{(\mu)}_{i,s}$ is bounded for $s \in [\rho/4,\rho/2]$ and for $i=1,2$. Indeed, since $v\in L^\infty (B_\rho)$ and $\theta$ is bounded on $[s/4,s/2]$, we deduce from \eqref{eq asumption growth} and  Lemma \ref{lemma second polynomial scalinginv bounds} that	for  $s \in [\rho/4,\rho/2]$
		\begin{equation*}\label{eq bound large s}
			\Vert Q_{1,s} \Vert_{L^{\infty}(B_s)}s^{\theta_+}+\Vert Q_{2,s} \Vert_{L^{\infty}(B_s)} s^{\theta_-+1}\leq C_\rho s^{\beta -1  +\theta_+}+\Vert v\Vert_{L^\infty(B_\rho)}.
	\end{equation*}
From here, the boundedness claim readily follows.\\

	 On the other hand, if we take $J$ such that $2^J r \in [\frac{\rho}{4}, \rho/2)$, we can use the montonicity of $\theta$ combined with \eqref{eq telescopic} to obtain the estimate	
	\begin{eqnarray*}
		\frac{|q^{(\mu)}_{1,r}-q^{(\mu)}_{1,2^Jr}|}{\theta(r)} \leq  C \sum_{l=0}^{J-1} \frac{\theta(\rho2^{l-J-1})}{\theta(r)}(2^{l-J})^{\beta-|\mu|-1} \leq  C \sum_{l=1}^{J} \frac{\theta(\rho2^{-i-1})}{\theta(r)}(2^{-i})^{\beta-|\mu|-1}.
	\end{eqnarray*}
So, since, $|\mu| < {\beta-1}$ and \eqref{eq boundedness ply} holds, we deduce from the dominated convergence theorem that the right hand side of the latter inequality goes to zero as $r\to 0^+$. By applying an analogous reasoning to \eqref{eq telescopic 2}, we obtain \eqref{eq boundedness ply}.\\
	
	Focusing on \eqref{eq growthconclusion}, let us notice that the monotonicity of $\theta$ combined with \eqref{eq telescopic} and \eqref{eq telescopic 2} yields
		\begin{equation}\label{eq bound difference 1}
		|q^{(\mu)}_{1,r}-q^{(\mu)}_{1,2^Jr}|\leq C \theta(r) r^{\beta-|\mu|-1} \sum_{l=0}^{J-1} (2^l)^{\beta-|\mu|-1}\leq C \theta(r)(2^Jr)^{\beta-|\mu|- 1},
	\end{equation}
	and
	\begin{equation}\label{eq bound difference}
			|q^{(\mu)}_{2,r}-q^{(\mu)}_{2,2^Jr}|\leq C \theta(r)(2^Jr)^{\beta-2+\theta_+-\theta_--|\mu|},
	\end{equation}
respectively. \\

	It follows that for any $R>1$,		
	\begin{equation}\label{eq polydiff growth 1}
		\Vert (Q_{1,r }-Q_{1,R r })d^{\theta_+}\Vert_{L^\infty(B_{Rr})}+\Vert (Q_{2,r }-Q_{2,R r })d^{\theta_-+1}\Vert_{L^\infty(B_{Rr})}\\\leq C \theta(r) (Rr)^{\beta-1+{\theta_+}}.
	\end{equation}

	Hence, using the triangular inequality and combining it with \eqref{eq asumption growth} and \eqref{eq polydiff growth 1} 
	\begin{eqnarray*}
		r^{{\theta_+}+\beta-1}\theta(r)	\Vert v \Vert_{L^\infty(B_{rR})}
		&=&\Vert v-Q_{1,r}d^{\theta_+}-Q_{1,2,r}d^{\theta_-+1}\Vert_{L^\infty(B_{Rr})}\\
		&\leq& \Vert v - Q_{1,r R}d^{\theta_+}-Q_{2,rR}d^{\theta_-+1}\Vert_{L^\infty(B_{Rr})}\\
		&& + \Vert (Q_{1,r }-Q_{1,R r })d^{\theta_+}\Vert_{L^\infty(B_{Rr})}\\
		&&+  \Vert (Q_{2,r }-Q_{2,R r })d^{\theta_-+1}\Vert_{L^\infty(B_{Rr})}\\
		&\leq&  C\theta(Rr)(Rr)^{\beta-1+{\theta_+}}+C\theta(r) (Rr)^{\beta-1+{\theta_+}}.
	\end{eqnarray*}
	Therefore, by the monotonicity of $\theta$, \eqref{eq growthconclusion} follows.\\
	
	We conclude the proof of the Lemma showing \eqref{eq laplacian linear terms}. Let $f(x) = Q_{1,r}(x)d^{\theta_+}( x)+Q_{2,r}(x)d^{\theta_-+1}( x)$, computing directly we have that
	
	\begin{eqnarray}\notag
		\Delta f(x) &=& Q_{1,r}(x) \Delta (d^{\theta_+}(x))+Q_{2,r}(x)\Delta (d^{\theta_-+1}( x))\\\notag
		&&+2 \theta_+ d^{\theta_+-1}\nabla Q_{1,r}(x) \cdot \nabla d(x)+2(\theta_-+1)d^{\theta_-}\nabla Q_{2,r}(x)\cdot \nabla d(x)\\\label{eq lapaux}
		&&+ d^{\theta_+}(x) \Delta Q_{1,r}(x)+ d^{\theta_-+1}( x) \Delta Q_{2,r}(x)
	\end{eqnarray}
	
	 Since  $\Vert \Omega \cap B_{2}\Vert_{C^{\beta}}\leq 1$, Lemma \ref{lemma generalized distance} provides the expansion 
	\begin{equation}\label{eq aprox distlemma}
		d(x)= P(x)+g(x) = \sum_{|\lambda|\leq \lfloor \beta\rfloor} p^{(\lambda)}x^\lambda +g(x),
	\end{equation} 
	where $P \in \mathbf{P}_{\lfloor \beta \rfloor}$ with $\Vert P\Vert_{L^\infty(\Omega \cap B_1)} \leq C_\beta$, and where $g \in C^{\infty}(\Omega \cap B_2)$ satisfies $|\nabla g| \leq C_\beta|x|^{\beta-1}$. Therefore, 
	\begin{eqnarray}\notag
	\frac{r^2 d^{\theta_+-1}}{\theta(r) r^{\theta_+ +\beta-1}} \nabla Q_{1,r}(rx) \cdot \nabla d(rx) &=& \frac{r^2 d^{\theta_+-1}}{\theta(r) r^{\theta_+ +\beta-1}}\left(\nabla Q_{1,r}(rx)\cdot \nabla P(rx) + \nabla Q_{1,r}(x)\cdot \nabla g(rx)\right)\\\label{eq aux1}
	&&+ R_{1,r}(x)\Big(\frac{d(rx)}{r}\Big)^{\theta_+-1}+ m_1(x,r)
\end{eqnarray}
where  $R_{1,r} \in \mathbf{P}_{2\lfloor \beta \rfloor-3}$ and $|m_1(x,r)|\leq Cm(r)$ with $m$ only depending on $\beta$ and $\kappa$ and such that $m(r)\to 0$ as $r\to 0^+$. Analogously,
\begin{eqnarray}\notag
	\frac{r^2 d^{\theta_-}}{\theta(r) r^{\theta_+ +\beta-1}} \nabla Q_{2,r}(rx) \cdot \nabla d(rx) &=& \frac{r^2 d^{\theta_-}}{\theta(r) r^{\theta_+ +\beta-1}}\left(\nabla Q_{2,r}(rx)\cdot \nabla P(rx) + \nabla Q_{2,r}(x)\cdot \nabla g(rx)\right)\\\label{eq aux2}
	&&+ R_{1,r}(x)\Big(\frac{d(rx)}{r}\Big)^{\theta_-}+ m_2(x,r)
\end{eqnarray}
where  $R_{2,r} \in \mathbf{P}_{\lfloor 2\beta+\theta_+-\theta_- \rfloor-4}$ and $|m_2(x,r)|\leq Cm(r)$.

In a similar vein, using \eqref{eq aprox distlemma}, we deduce that 
\begin{eqnarray}\notag
	\frac{r^2}{\theta(r) r^{\theta_+ +\beta-1}}\Big( d^{\theta_+}(rx) \Delta Q_{1,r}(r x)+ d^{\theta_-+1}(r x) \Delta Q_{2,r}(r x)\Big) = \\\label{eq aux3}
	m_3(x,r) + R_{3,r}(x) \Big(\frac{d(rx)}{r}\Big)^{\theta_+-1}+R_{4,r}(x) \Big(\frac{d(rx)}{r}\Big)^{\theta_-},
	\end{eqnarray}
	with $R_{i,r} \in \mathbf{P}_{M}$ for $i=3,4$ for some $M \in \N$ only depending on $\beta$ and $\kappa$ and where $|m_3(x,r)|\leq Cm(r)$.

	On the other hand, thanks to Lemma \ref{lemma generalized distance}, we also have that
	\begin{equation}\label{eq almost sol large beta 0}
	\Delta d^{\theta_+}(x) = \frac{\kappa}{d^2(x)} d^{\theta_+} + d^{\theta_+-1} \Big(P_0(x)+g_2(x)\Big),
\end{equation}
with $P_0 \in \mathbf{P}_{M}$ for some $M$ depending on $\beta$ and $|g_2(x)| \leq C|x|^{\beta-1}$. Combining \eqref{eq lapaux}, \eqref{eq aux1}, \eqref{eq aux2}, \eqref{eq aux3}, \eqref{eq almost sol large beta 0}, and \eqref{eq lap distpower} applied to $d^{\theta_-+1}$, we deduce
\begin{eqnarray*}
 	\frac{r^2}{\theta(r) r^{\theta_+ +\beta-1}}\Big(	\Delta f(rx) -\frac{\kappa}{d(rx)^2} f(rx)\Big)  =P_{1,r}(x) \Big(\frac{d(rx)}{r}\Big)^{\theta_+-1}+P_{2,r}(x) \Big(\frac{d(rx)}{r}\Big)^{\theta_-} + m_*(r,x)
\end{eqnarray*}
with $P_{i,r} \in \mathbf{P}_{M}$ for $i=1,2$ for some $M \in \N$ only depending on $\beta$ and $\kappa$ and where $|m_*(x,r)|\leq Cm(r)$. This proves \eqref{eq laplacian linear terms}.
\end{proof}

\bibliographystyle{abbrv}

\end{document}